\documentclass[10pt,notitlepage,twoside,a4paper]{amsart}

\usepackage{amsmath,amssymb,enumerate}

\usepackage{epsfig,fancyhdr,color}%,showkeys,amsmidx

\usepackage{amssymb}
\usepackage{amsmath,amsthm}
\usepackage{latexsym}
\usepackage{amscd}
\usepackage{psfrag}
\usepackage{graphicx}
\usepackage[latin1]{inputenc}
\usepackage[all]{xy}
\usepackage[mathcal]{eucal}
\usepackage{cancel}

\definecolor{NoteColor}{rgb}{1,0,0}

\renewcommand{\textsc}{\textcolor{red}}

\newtheorem{theorem}{\rm\bf Theorem}[section]
\newtheorem{proposition}[theorem]{\rm\bf Proposition}

\newtheorem*{theorem 1}{\rm\bf Proposition 1}
\newtheorem*{theorem 2}{\rm\bf Proposition 2}

\theoremstyle{definition}
\newtheorem{definition}[theorem]{\rm\bf Definition}

\theoremstyle{remark}

\def\interieur#1{\mathord{\mathop{\kern 0pt #1}\limits^\circ}}

\title[Timelike Hilbert geometry of the spherical simplex]{Timelike Hilbert geometry of the spherical simplex}

\author{Athanase Papadopoulos}

\address{Athanase Papadopoulos,  Universit{\'e} de Strasbourg and CNRS,
7 rue Ren\'e Descartes,
 67084 Strasbourg Cedex, France}
\email{athanase.papadopoulos@math.unistra.fr}

\author{Sumio Yamada}

 \address{Sumio Yamada,
  Gakushuin University, 1-5-1 Mejiro, Toshima, Tokyo, 171-8588, Japan}
 \email{yamada@math.gakushuin.ac.jp}

\date{\today}

\begin{document}

  \maketitle
  
  \centerline{\emph{To Norbert A'Campo with friendship and admiration}} 
  \begin{abstract} We first review the notion of  timelike metric spaces. This is a  metric theory developed by Herbert Busemann, as a geometrical setting for the theory of general relativity. We review in particular the notions of timelike Hilbert and timelike spherical Hilbert geometry. We prove the following result on the timelike spherical Hilbert geometry of simplices: 
 Let $\Delta_2$ be a simplex on the 2-sphere and $\tilde{\Delta}_2$ the antipodal simplex. We show that the timelike spherical Hilbert geometry associated with the pair  $\Delta_2, \tilde{\Delta}_2$ is isometric to a union of six copies of vector spaces equipped with a timelike norm, isometrically and transitively acted upon by the group $\mathbb{R}_{>0}^2 \times \mathbb{Z}_3\times \mathbb{Z}_2$. This is a timelike spherical analogue of a well-known result (due to Busemann) stating that the Hilbert metric of a Euclidean simplex is isometric to a metric induced by a normed vector space.  At the same time, this gives a new example of timelike space.
  \medskip

  
  \medskip
  
  \end{abstract}
     
  \medskip

  \tableofcontents

  \section{Introduction}\label{s:Intro}
  
  In his paper \cite{B-Timelike}, Herbert Busemann introduced the notion of timelike space. This is a Hausdorff topological space $\Omega$ equipped with a partial order relation $<$  and a function $d$ called \emph{(timelike) distance function}\index{timelike distance function}---even though its does not satisfy the usual axioms of a distance---, such that for every $x$ and $y$ in $\Omega$, $d(x,y)$ is defined if and only if $x\leq y$ (that is, $x<y$ or $x=y$) and such that for any triple $x, y, z$ in $\Omega$ satisfying $x\leq y\leq z$,  the reverse triangle inequality, also called  the \emph{time inequality}),\index{time inequality} holds, that is, we have 
  \[d(x,z)\geq d(x,y)+d(y,z).\]
  
A classical example where this time inequality holds for triples of points is the setting of the Lorentzian space of special relativity (we shall recall the definition below). 
    
  In Busemann's general setting, there are several axioms that the partial order relation $<$ and the distance function $d$ satisfy---too numerous to be mentioned all here; we shall say a few words on these axioms in the next section. Note that unlike the usual distance functions, the timelike distance function $d$ does not define a topology on the underlying space $\Omega$; this is why such a topology is given in advance.
  
Busemann's theory of timelike spaces is a general metric setting including as a special case semi-Riemannian (also called pseudo-Riemannian) manifolds, which, in turn, generalize Riemannian manifolds in the sense that the quadratic form defining the metric is infinitesimally nondegenerate but not necessarily positive definite. An example of a timelike space is the $n$-dimensional Lorentzian space. This is the timelike space in which the underlying topological space $\Omega$ is the real affine $n$-dimensional space $A^n$ equipped with the indefinite metric given in the affine coordinates $\mathbf{x}=(x_1,\ldots, x_n)$ by
\[
\lambda_n(\mathbf{x})=x_1^2-\sum_{i=2}^nx_i^2.\]
The associated order relation $<$ is defined for  $\mathbf{x}=(x_1,\ldots, x_n)$ and  $\mathbf{y}=(y_1,\ldots, y_n)$ as
\[ \mathbf{x} < \mathbf{y} \iff
 (x_1<y_1 \ \hbox { and } \ \lambda_n(\mathbf{x}-\mathbf{y})>0).
\]
The timelike distance function is then given by
\[d(\mathbf{x},\mathbf{y})=\sqrt{\lambda_n(\mathbf{x}-\mathbf{y})}.\]

More general examples of timelike spaces are the timelike Minkowski spaces. These are timelike spaces in which the underlying space is a finite-dimensional vector space equipped with a distance induced by a timelike norm, that is, a function analogous to a usual norm that defines a Minkowski space (finite-dimensional normed space) except that this function (the ``Minkowski functional"), instead of being defined on the entire vector space, is only defined on a proper cone (that is, a cone that does not contain any line) with apex the origin, and instead of being convex, is concave. In other terms, the unit ball of  a timelike Minkowski space  is contained in a proper cone which, seen from the origin, is a concave hypersurface.
The timelike Minkowski spaces are the local models of the continuously differentiable timelike Finsler manifolds\index{timelike!Finsler manifold} in much the same way as the classical Minkowski spaces are the local models for  a continuously  differentiable Finsler manifolds (see \cite{Busemann-Minkowski}), and in the same way as the Euclidean vector spaces are the local models for Riemannian manifolds.

  The motivation for studying timelike metrics is that they form a geometric setting for  the theory of general relativity, in much the same way as Lorentzian space forms a setting for special relativity. The order relation inherent in the definition of a timelike metric is a mathematical abstraction of the causality property of space-time of relativity theory; that is, $x<y$ reflects the fact that $y$ is in the future of $x$, or that there is a causality relation between $x$ and $y$ (or that $y$ is influenced by $x$).

Since we talked about general relativity, let us also mention Einstein's important article on the subject, \cite{Einstein}.
To close the reference to physics, let us mention that a general metric setting for general relativity, called \emph{Chronogeometry},\index{chronogeometry} was developed, independently from that of Busemann, by A. D. Alexandrov and his school in Russia, see e.g. the review in \cite{BP-Chrono}.

Let us mention some examples. Busemann, in \cite{B-Timelike}, showed that products of timelike spaces with metric spaces satisfying certain hypotheses are timelike spaces,  for some appropriate definition of the product metric \cite[\S 4]{B-Timelike}. Busemann's work on this subject is also reviewed in \cite{PY}. Other interesting examples of timelike spaces are the exterior (also called timelike) Funk and Hilbert geometries of convex subssets of $\mathbb{R}^n$. Based on Busemann's ideas, we have developed the theory of exterior Funk and Hilbert geometries in the paper \cite{PY-Timelike-Hilbert} in which we also introduced non-Euclidean variants of these geometries. In the latter setting,  the underlying convex sets are contained in the $n$-dimensional sphere or hyperbolic space rather than in the Euclidean space. In the same paper, we gave a characterization of the classical de Sitter geometry as a special case of a timelike spherical Hilbert geometry, namely, it is the exterior Hilbert geometry of a union of two disjoint geometric antipodal discs in the sphere.

  Let us point out that a general theory of pseudo-Finslerian manifolds,\index{pseudo-Finslerian manifold} in the setting of timelike spaces, has been developed in the recent PhD  thesis of Guilllaume Buro, see \cite{Buro-these}. In particular, the author proves there classical and important theorems of Finsler geometry (Busemann--Mayer, Hopf--Rinow, etc.)  in this pseudo-Finslerian setting.

  In the present paper, we study the case of a timelike spherical Hilbert geometry which is the exterior geometry of two antipodal simplices on the sphere. We prove the following: 
   
   \medskip
   
 \noindent{\textbf{Theorem 6.1} \textit{Let $\Delta_2$ be a spherical $2$-dimensional simplex in the sphere $S^2$, and $\tilde{\Delta}_2$ its antipodal simplex. Then the timelike Hilbert geometry of $\Omega= S^2\setminus ( \Delta_2 \cup \tilde{\Delta}_2)$ is isometric to a union of  $six$ copies of timelike normed spaces  on which the abelian group $\mathbb{R}_{> 0}^2  \times \mathbb{Z}_{3} \times  \mathbb{Z}_{2} $ acts isometrically and transitively. 
}
 
\medskip

The higher-dimensional cases can be treated similarly.  The exposition, however, needs extra technicalities, and it will be given in a separate article.    
  
  Theorem \ref{th:1} may be compared with a result of Busemann saying that, in the classical (non-timelike) case, the Hilbert geometry of a simplex in any dimension is induced by a norm on a finite-dimensional vector space. (In fact, Busemann's result  is   more precise; see   \cite[p. 35]{B-Timelike} and \cite[p. 313]{BP}). In dimension 2, our result is also a timelike analogue of a result of Phadke  on the Hilbert geometry of a triangle \cite{Phadke}.

 \section{The timelike Hilbert geometry of a spherical simplex} \label{s:timelike-Hilbert}
We first briefly review the notion of timelike space and of timelike Hilbert geometry. 

A timelike space is a Hausdorff topological space $\Omega$ equipped with a partial order relation $<$ and a function $d$ which plays the role of a distance function. 
 The distance $d(x,y)$  is defined only for pairs $(x,y)\in \Omega\times \Omega$ satisfying $x\leq y$ (that is, either $x=y$ ot $x<y$) and it satisfies the following three axioms:
 \begin{enumerate}
 \item $d (x,x)=0$ for every $x$ in $\Omega$;
 \item $d(x,y) \geq 0$ for every $x$ and $y$ in $\Omega$ such that $x<y$;
 \item  $d(x,y)+d(y,z)\leq d(x,z)$ for all triples of points $x,y,z$ in $\Omega$ satisfying $x<y<z$.
 \end{enumerate}
 The last property is the \emph{time inequality}. It is a reverse triangle inequality.  
 We note that the distance function $d$ is asymmetric, that is, $d(x,y)$ is not necessarily equal to $d(y,x)$. (In fact, if $d(y,x)$ is defined, $d(x,y)$ is not defined unless $x=y$.)

The notion of timelike space, as introduced by Busemann in   \cite{B-Timelike}, involves a list of additional axioms which make relations between the topological properties of the space $\Omega$, the distance function $d$ and the partial order relation $<$. 
For instance, one axiom requires that any neighborhood of a point $q$ in $\Omega$ contains points $x$ and $y$ satisfying $x<q<y$.  Other axioms insure the existence of local geodesic segments and the uniqueness of their extension. These properties are timelike analogues of the properties of a G-space in the sense of Busemann (G stands for ``Geodesic"), a theory developed in \cite{G}. 
 We do not state these axioms in the present paper because there are too many of them. They are all satisfied in the special case of the timelike spherical Hilbert  geometry which we consider here.  

We shall recall  the notion of timelike spherical Hilbert metric  studied in \cite{PY-Timelike-Hilbert}, since our main result concerns this metric.  Before that, we need to recall another timelike metric, namely, the timelike Funk metric. First of all, we remind the reader of some elementary facts on convex subsets of the sphere. The basic notions in spherical convexity theory that we use are presented in detail in the paper \cite{PY-Timelike-Hilbert}.

For $n\geq 1$, let $S^n$ be the $n$-dimensional sphere, which we regard as the unit sphere in $\mathbb{R}^{n+1}$. We denote by $O$ its center.

A \emph{spherical segment} of $S^n$ is a segment of a great circle of this sphere. Such a segment is the shortest path between its endpoints if and only if it is contained in a hemisphere, that is, the complement in  $S^n$ of the intersection of this sphere with a hyperplane passing through the origin.

A subset $I\subset S^n$ is said to be \emph{convex} if $I\not=S^n$ and if any two points in $I$ are joined by a shortest line contained in $I$. Note that this implies that $I$ is contained in an open hemisphere. Notions such as \emph{hypersphere},  \emph{great hypersphere}, \emph{open} and \emph{closed hemisphere}, \emph{pole} of a hemisphere, etc. are defined in a natural way. The following notion plays a crucial role in this setting:

 A \emph{supporting hyperplane} $\pi$  to an open convex subset $I$ of $S^n$ is a great hypersphere whose intersection with the closure $\overline{I}$ of $I$ is nonempty and such that $I$ is contained in one of the two connected components of the complement of $\pi$ in $S^n$. 
For each point on the boundary of a convex subset of the sphere, there is a supporting hypersphere containing it. 

Let  $I_1$ and $I_2$ be two open convex subsets of $S^n$ and let $K_1$ and $K_2$ be respectively the hypersurfaces that bound them. In other words, for $i=1,2$, $K_i= \overline{I}\setminus I$.  
%
%We shall also say that a supporting hyperplane to $I_i$ is a supporting hyperplane to $K_i$ or to $\overline{K_i}$, depending on the subset of the sphere that we want to stress on.

 \begin{definition}\label{rel-pos}
 The two hypersurfaces $K_1,K_2$ are said to be \emph{in good position} if the following two properties are satisfied: 
 \begin{enumerate}
 \item $\overline{I_1}\cap \overline{I_2}=\emptyset$;
 \item For any great circle $C$ satisfying $C\cap K_i\not=\emptyset$ for $i=1,2$, the set $C\setminus (\overline{I_1}\cup \overline{I_2})$ is the union of two spherical segments of length $<\pi$. 
 \end{enumerate}
 \end{definition}
 
 The following is proved in \cite{PY-Timelike-Hilbert} (Proposition  11.3).
  
 \begin{proposition} \label{prop:anti} Assume $K_1, K_2$ are in good position. Then, the union  $I_1\cup I_2$ contains a pair of antipodal points, each contained in one of the sets $I_1, I_2$.
 \end{proposition}
  
%\begin{figure}[!ht] 
%\centering
%\includegraphics[width=0.55\linewidth]{Figure_1.eps}    \caption{\small {Timelike spherical Hilbert geometry}   \label{domain Omega}  
%\end{figure}

\begin{figure}
\label{fig1}
\centering
 \includegraphics[width=1\linewidth]{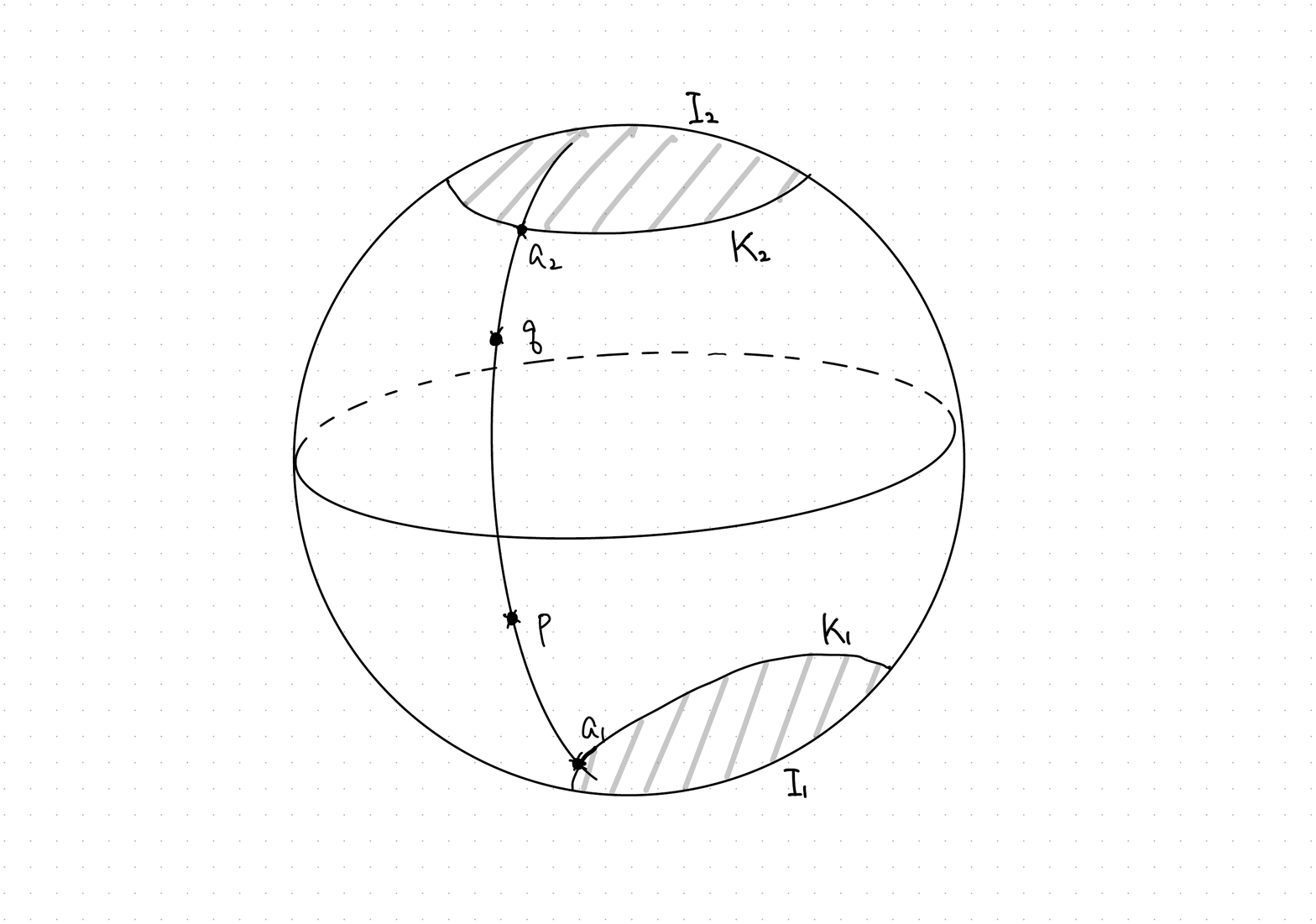} 
    \caption{\small {The timelike spherical Hilbert geometry associated with two hypersurfaces $K_1$ and $K_2$}}
\end{figure}

 In the following, we shall assume that our two hypersurfaces $K_1$ and $K_2$ are in good position. Let $\Omega= S^2\setminus (\overline{I_1}\cup \overline{I_2})$. We define a partial order relation on $\Omega$.
 
 \begin{definition}[Partial order] \label{def:Partial} Given two points $p$ and $q$ in $\Omega$,  we say that \emph{$q$ is in the future of $p$}, or that \emph{$p$ is in the past of $q$}, and we write $p<q$, if there exists a segment $[p,q]$ of a great circle $C$ in $\Omega$  joining $p$ and $q$ and if there exist two points $a_1\in C\cap K_1$ and $a_2\in C\cap K_2$ such that the four points $a_1,p,q,a_2$ lie in that order on $C$, with $]a_1,a_2[\subset \Omega$ and such that $]a_1,a_2[$ is not contained in any supporting hyperplane to $K_1$ or $K_2$.  (See Figure 1.)   

 \end{definition}

 We denote  the set of pairs $(p, q)$ in $\Omega \times \Omega$ satisfying $p < q$ by $\Omega_<$ and, as usual, we write $p \leq q$ when $p<q$ or $p=q$.  
 We have the following (Proposition 11.8 of \cite{PY-Timelike-Hilbert}).
 
  \begin{proposition}[Transitivity of the partial order relation]  Let $p$, $q$ and $r$ be three points in $\Omega$ satisfying $p\leq q$ and $q\leq r$. Then we have $p\leq r$.
 \end{proposition}

  We now define the \emph{timelike spherical relative Funk distance}\index{timelike!relative Funk metric!spherical}  $F_1^2$ on the subset $\Omega_{\leq}$ of the product $\Omega\times \Omega$  consisting of pairs $(p,q)$ with $p \leq q$ by 
 \[F_1^2(p,q)=\log\frac{\sin d(p,b(p,q))}{\sin d(q,b(p,q))},\]
 where  $d$ is the usual spherical distance and where for the two given points $p$ and $q$, $b(p, q)$ denotes the point where the geodesic ray from $p$ through $q$ intersects $K_2$. (This point is equal to $a_2$ in Definition \ref{def:Partial} above.)
We then extend this definition to the pairs $(p,p)$ in the diagonal of $\Omega\times\Omega$ by setting $F_1^2(p,p)=0$ for any such pair. Note here that we require that the great circle $C$ on which $p$ and $q$ lie intersects both $K_1$ and $K_2$, even though in order to define $F_1^2$,  $C$ does not need to intersect $K_1$.
 
The function $F_1^2(p,q)$ satisfies the timelike inequality (Proposition 11.12 of \cite{PY-Timelike-Hilbert}). It defines a timelike Finsler metric.

Next, we define the \emph{timelike  spherical relative reverse Funk metric} $\overline{F_1^2}$ associated with the pair $(K_1,K_2)$. For this, we first consider the timelike spherical relative Funk metric $F_2^1$ associated with the ordered pair $(K_1,K_2)$, and we define the new function $\overline{F_1^2}$, whose domain of definition  is equal to the domain of definition of $F_1^2$, by
 \[\overline{F_1^2}(p,q)=F_2^1(q,p) .\] 
 
 Now we can define the timelike spherical Hilbert metric:
 
 \begin{definition}[Timelike spherical Hilbert metric]
 The timelike spherical Hilbert\index{timelike!Hilbert metric!spherical}  metric $H_1^2$ associated with the ordered pair $(K_1,K_2)$ is defined on the set of ordered pairs $(p,q)$ such that $p<q$ by the formula 
 \[ H(p,q)=\frac{1}{2}(F_1^2(p,q)+\overline{F_1^2}(p,q)).\]
 
 In this setting, the convex set $K_1$ represents the past and the convex set $K_2$ the future.  (See Figure 2a.)
 
%\begin{figure}[!ht] 
%\centering
%\includegraphics[width=0.55\linewidth]{Figure_2a.eps, Figure_2b.eps}    \caption{\small {Hilbert $=$ relative Funk symmetrized}   \label{hilbert metric}  
%\end{figure}

\begin{figure}
\label{fig2}
\centering
 \includegraphics[width=1\linewidth]{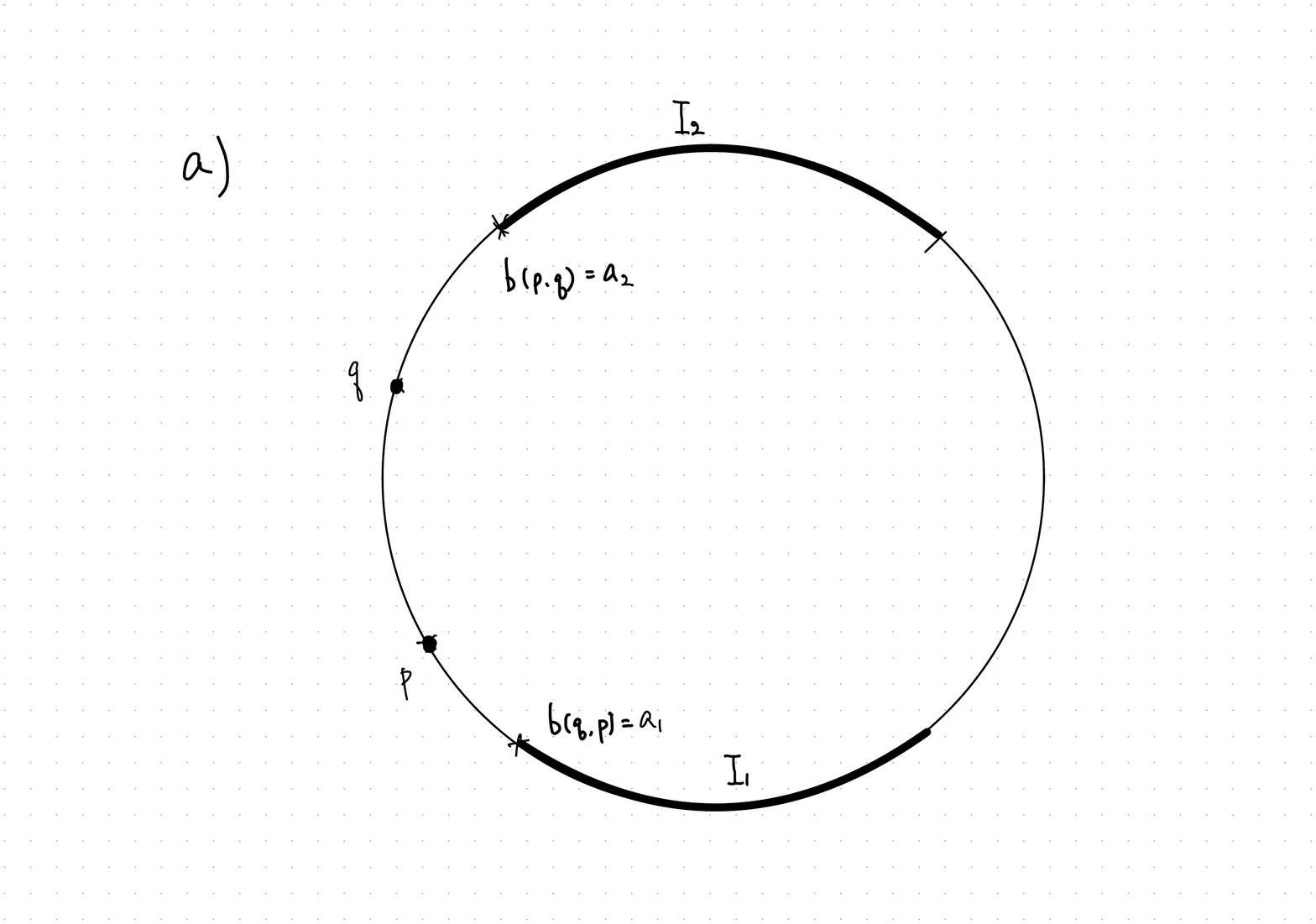}  
 \includegraphics[width=1\linewidth]{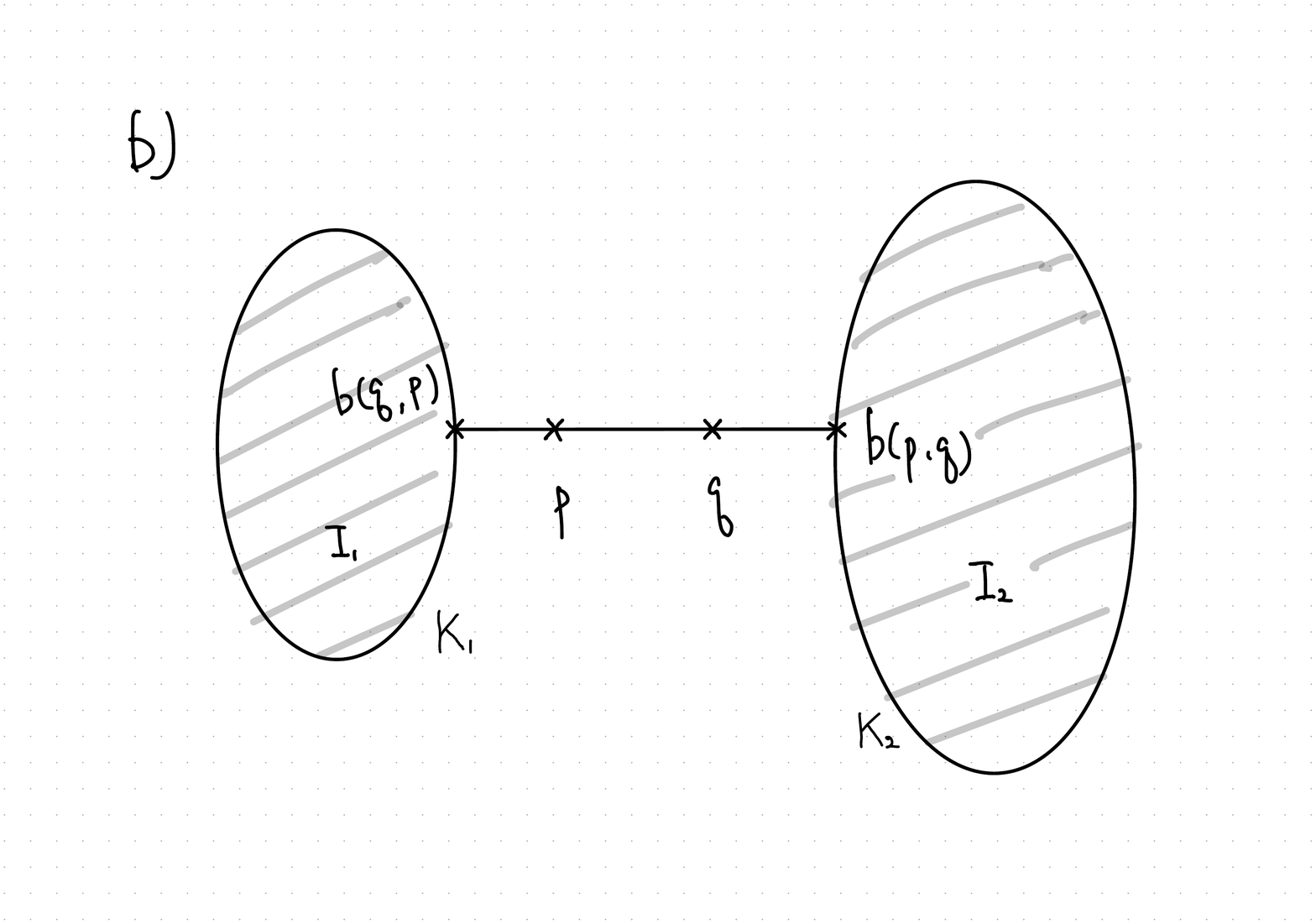}  
   \caption{\small {The Hilbert metric is the relative Funk metric symmetrized}}   
\end{figure}

As usual, the definition is extended to the case where $p=q$ by setting in this case $H(p,q)=0$.
 \end{definition}

 We recall that given four points $p_1, p_2, p_3, p_4$ situated in that order on a great circle on the sphere, their \emph{spherical cross ratio}\index{spherical cross ratio} is defined by
\[
[p_1, p_2, p_3, p_4] = \frac{\sin d(p_2, p_4) \sin d(p_3, p_1)}{\sin d(p_3, p_4) \sin d(p_2, p_1)}.
\] 
Its values are in ${\mathbb R}_{\geq 0} \cup \{\infty\}$. The spherical cross ratio is a projectivity  invariant. The meaning is explained in \cite[p. 243]{2012-Hilbert3}. Note that with the introduction of the spherical cross ratio, the projective transformations of the sphere become precisely those that preserve the cross ratio of four aligned points.

 For any pair of points $(p, q)$ in $\Omega_<$, let $a_1 \in K_1$ and $a_2 \in K_2$ be the intersection points between the great circle 
through $p$ and $q$ and the two hypersurfaces $K_1$ and $K_2$, such that $a_1, p, q, a_2$ lie in that order on the arc of great circle $[a_1, a_2] \subset \Omega$.   With this notation, the timelike spherical Hilbert distance associated with the pair $(K_1,K_2)$  is also given by the following formula:
\begin{proposition}\label{prop:SH2} Let $p$ and $q$ be two points in $\Omega$ satisfying $p<q$ and let $[a_1,a_2]$ be the segment of great circle containing $p$ and $q$ with $[a_1,a_2]\cap K_i=a_i$  for $i=1,2$. Then, we have:
\[
H(p, q) =\frac{1}{2} \log [a_1, p, q, a_2].
\]
\end{proposition}
The following is Proposition 13.3 of \cite{PY-Timelike-Hilbert}:

\begin{proposition}[Invariance] \label{prop:SHinv}  
The timelike spherical Hilbert metric associated with the pair of convex hypersurfaces $K_1,K_2\subset S^n$ is invariant by the projective transformations of the sphere $S^n$ that preserve setwise each of the two convex hypersurface $K_1,K_2$. \end{proposition}

The timelike spherical Hilbert metric $H$ has an underlying timelike Finsler structure which we describe in the paper  \cite{PY-Timelike-Hilbert}. Furthermore, Proposition 13.4 of \cite{PY-Timelike-Hilbert} gives a characterization of a natural class of geodesics for this metric associated with any ordered pair of convex hypersurfaces $(K_1,K_2)$, namely, the  spherical segments of the form $]a_1,a_2[$, equipped with their natural orientation from $a_1$ to $a_2$ and satisfying the following three properties
\begin{enumerate}
\item \label{pr1} $a_1 \in K_1$ and $a_2 \in K_2$; 
\item $]a_1,a_2[$  is not contained in any supporting hyperplane to $K_1$ or to $K_2$;
\item  \label{pr3} the open spherical segment $]a_1,a_2[$ is in the complement of $K_1\cup K_2$ .
\end{enumerate}

The following proposition is  Proposition  13.4 of \cite{PY-Timelike-Hilbert}.

\begin{proposition}\label{Spherical-Hibert-geo}
(a) In a timelike spherical Hilbert geometry $(\Omega, H)$ associated with an ordered pair of convex hypersurfaces $(K_1,K_2)$, the spherical segments of the form $]a_1,a_2[$, equipped with their natural orientation from $a_1$ to $a_2$ and satisfying the following three properties
\begin{enumerate}
\item \label{pr1} $a_1 \in K_1$ and $a_2 \in K_2$; 
\item \label{pr2}  $]a_1,a_2[$  is not contained in any supporting hyperplane to $K_1$ or to $K_2$;
\item  \label{pr3} the open spherical segment $]a_1,a_2[$ is in the complement of $K_1\cup K_2$ 
\end{enumerate}
 are $H$-geodesics with their given orientation. Each such geodesic is isometric to the real line.
 
 (b) The oriented spherical segments contained in the segments  of the form $[a_1,a_2]$  satisfying the properties (\ref{pr1}-\ref{pr3}) above are the unique $H$-geodesics
if and only if the following holds: There are no spherical geodesic segments $[a_1,a_2]$ satisfying properties (\ref{pr1}-\ref{pr3}) with $a_1$ in the interior of an open nonempty spherical segment $J_1\subset K_1$ and $a_2$ in the interior of an open nonempty segment $J_2\subset K_2$, such that  $J_1$ and $J_2$ are coplanar, i.e., contained in a 2-dimensional sphere.
\end{proposition}
 
The reader will notice the formal similarities with the properties and characterizations of geodesics in the classical Hilbert geometry; cf. also the characterization of geodesics in the timelike Euclidean Hilbert geometry recalled in the next section.

% 
%
%
%
%  \begin{definition}[Order relation]
%We introduce a partial order relation between points of $\Omega$.  For any two distinct points $p$ and $q$ in $\Omega$, we write
%\[
%p < q 
%\]
%if the following three properties are satisfied:
%\begin{enumerate}
%\item The Euclidean ray $R(p, q)$ from $p$ through $q$ intersects the hypersurface $K$;
% \item $R(p, q)$ does not belong to a supporting hyperplane of $K$;
% \item the closed Euclidean segment $[p,q]$ does not interesect $K$.
% \end{enumerate}
%\end{definition}
%When $p<q$, we say that $q$ \emph{lies in the future of} $p$ and that $p$ \emph{lies in the past of} $q$.   We write $p\leq q$ if either $p<q$ or $p=q$.
%
%We denote by $\Omega_<$ (resp. $\Omega_{\leq}$) the set of ordered pairs $(p,q)$ in $\Omega\times\Omega$ 
%satisfying $p<q$  (resp. $p\leq q$).   The set  $\Omega_<$ is disjoint from the diagonal set $\{(x, x) \,\,| \,\, x \in 
%\Omega\} \subset \Omega\times\Omega$.
   
 \section{Timelike  Euclidean Hilbert geometry induces timelike spherical Hilbert geometry}  
 
In the context of Hilbert geometry, a timelike spherical Hilbert metric can be represented as a timelike Euclidean Hilbert geometry in a natural manner (Section 9 of  \cite{PY-Timelike-Hilbert}).  Recall that the timelike Euclidean Hilbert geometry is realized in the complement of two non-intersecting open convex sets $I_1$ and $I_2$ bounded by convex hypersurfaces $K_1$ and $K_2$ respectively, 
 where the former is considered as the past set and the latter the future set.   Denote the complement of $I_1 \cup I_2$ by $\Omega$.
 The associated Euclidean timelike Hilbert distance $H_E(p, q)$ defined for $p$ and $q$ where $q$ lies in the future of $p$ is defined (Definition 9.1 of  \cite{PY-Timelike-Hilbert}) by taking the logarithm of Euclidean cross ratio,
 \[
 H_E(p, q) = \frac12 \log [a_1, p, q, a_2]. 
 \] 
where $a_1 \in K_1$ and $a_2 \in K_2$ are the intersection points of the line through $p$ and $q$ and the hypersurfaces $K_1$ and $K_2$ such that $a_1, p, q, a_2$ are aligned on the line segment $[a_1, a_2] \subset \Omega$ in that order.   
 
Proposition 9.2 of \cite{PY-Timelike-Hilbert} gives a characterization of a natural class of geodesics for this Euclidean timelike Hilbert  associated with an ordered pair of convex hypersurfaces $(K_1,K_2)$, namely, the Euclidean segments of the form $]a_1,a_2[$ equipped with their natural orientation from $a_1$ to $a_2$ and satisfying the following three properties
\begin{enumerate}
\item \label{pr1} $a_1 \in K_1$ and $a_2 \in K_2$; 
\item $]a_1,a_2[$  is not contained in any supporting hyperplane to $K_1$ or to $K_2$;
\item  \label{pr3} the open spherical segment $]a_1,a_2[$ is in the complement of $K_1\cup K_2$ .
\end{enumerate}
 
Now consider  the projection map $ \mathbb{R}^{n+1} \setminus \{0\} \rightarrow S^n$ defined  by sending $x$ to the intersection point of the ray $Ox$ and the unit sphere. This projection induces a map   
\[
\Phi: (\mathbb{R}^{n+1} \setminus \{0\}) \setminus (C_1 \cup C_2) \rightarrow S^n \cap (I_1 \cup I_2)     
\]
where $C_i$ is the open cone in $\mathbb{R}^{n+1}$ whose vertex is the origin $O$ and which is spanned by the open spherical region $I_i$.   Note then that $C_1$ and $C_2$ constitute a pair of open non-intersecting convex sets in $\mathbb{R}^{n+1}$, and the complement of their union $C_1 \cup C_2$ 
has a natural timelike Euclidean Hilbert metric $H_E$.  As the map $\Phi$ preserves the cross ratio (where it is understood that on the sphere, one takes the spherical cross ratio), we have the following (see Fig. \ref{fig3}a, \ref{fig3}b): 
 
\begin{theorem}
The map $\Phi$ is distance-preserving in the sense that
\[
H_E(x, y) = H(\Phi(x), \Phi(y))
\] for $x, y$ in $\mathbb{R}^{n+1} \setminus \{0\}$.  Each geodesic segment for $H_E$, represented by a Euclidean line segment,  is sent to a geodesic segment for $H$, that is, an arc of great circle.  

\end{theorem}

%\begin{figure}[!ht] 
%\centering
%\includegraphics[width=0.55\linewidth]{Figure_3a.eps, Figure_3b.eps}    \caption{\small {Euclidean vs Spherical models}   \label{cone construction}  
%\end{figure}

\begin{figure}
\label{fig3}
\centering
 \includegraphics[width=1\linewidth]{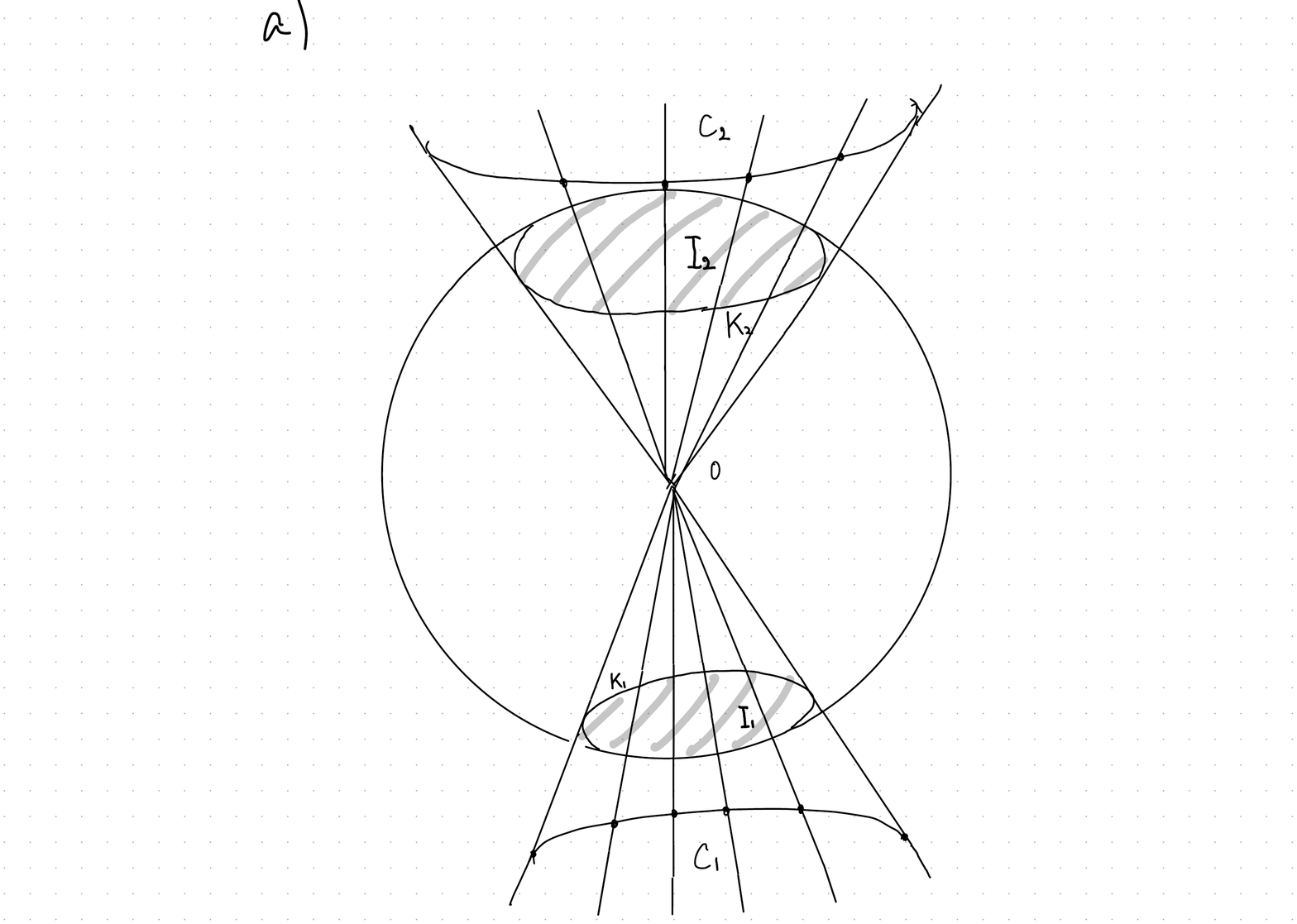}  
 \includegraphics[width=1\linewidth]{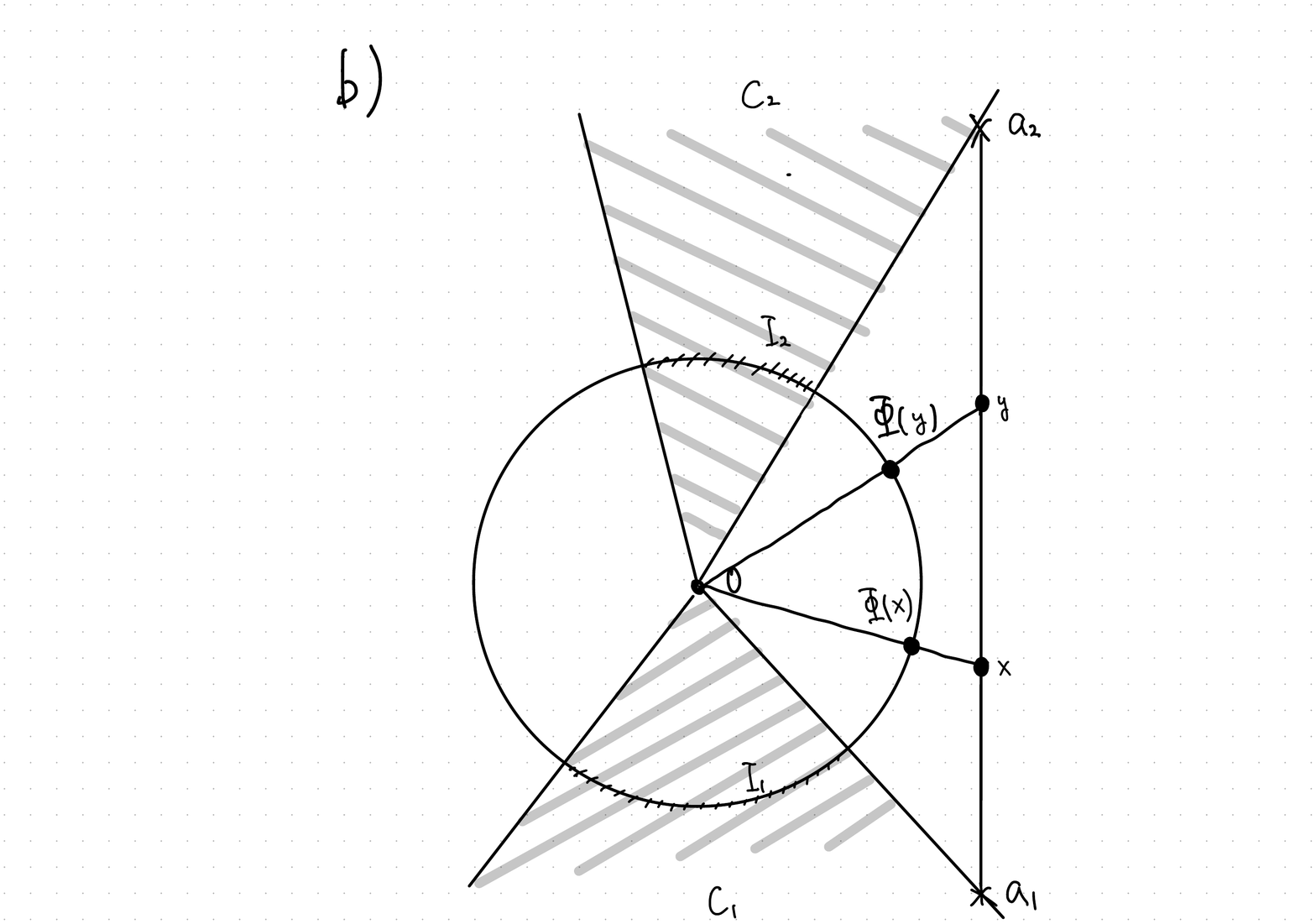}  
   \caption{\small {{Euclidean vs. spherical models}}}  
\end{figure}

 Note that the distance preservation occurs in the setting of the Perron-Frobenius Theorem (see, for example, \cite{Samelson}) where the projection is instead defined from the space of $n \times n$ matrices with positive components to the $(n^2-1)$-dimensional (spherical) simplex with the classical Hilbert metric,  
\[
\Phi: (\mathbb{R}_{>0})^{n^2} \rightarrow \mbox{ the first orthant} \cap S^{(n^2-1)}.
\]

   %
%\begin{definition}\label{prop:SH2} Let $p$ and $q$ be two points in $\Omega$ satisfying $p<q$ and let $[a_1,a_2]$ be the segment of great containing $p$ and $q$ with $[a_1,a_2]\cap K_i=a_i$  for $i=1,2$. Then, we have:
%\[
%d(p, q) =\frac{1}{2} \log [a_1, p, q, a_2].
%\]
%%\end{definition}
%
%This formula for the timelike spherical Hilbert geometry in terms of the cross ratio implies the following: 
% 
%\begin{proposition}[Invariance] \label{prop:SHinv}  
%The timelike spherical Hilbert geometry associated with the pair of convex sets $K_1,K_2\subset S^n$ is invariant by the projective transformations of the sphere $S^n$ that preserve setwise each of the two convex sets $K_1,K_2$. \end{proposition}

%The timelike spherical Hilbert geometry $(\Omega, d)$ has an underlying timelike Finsler structure which we describe in the next section. For that, we need first to talk about $H$-geodesics. 

\section{Simplical decomposition and geometric charts for $S^2$}
  \label{s:decomposition}

 The sphere $S^2$ is considered to be the unit sphere centered at the origin $O$ of $\mathbb{R}^{3}$ equipped with the Cartesian coordinate system 
 $
 x = (x_1, x_2 , x_{3}) 
 $.  
We define a \emph{standard simplex} to be  the intersection of $S^2$ with an \emph{orthant} of $\mathbb{R}^{3}$, that is,  a subspace of $\mathbb{R}^{3}$ defined by setting each of the coordinates to have a fixed sign. Thus,  we have eight standard spherical simplices, each one defined by a set of inequalities of the type 
\[
x_1 \lessgtr 0, x_2 \lessgtr 0, x_3 \lessgtr0.
\] 
Consider a pair of antipodal such simplices $\Delta_2$ and $\tilde{\Delta}_2$ in $S^2$,
\[
\Delta_2 = \{(x_1, x_2, x_3) \,|\, \forall j \,\, x_j >0 \mbox{ and } \sum_{i=1}^{3} (x_i)^2 = 1  \}
\]
and 
\[
\tilde{\Delta}_2 = \{(x_1, x_2,  x_3) \,|\, \forall j \,\, x_j <0 \mbox{ and } \sum_{i=1}^{3} (x_i)^2 = 1  \}.
\]

For each point $x$ of $S^2$ none of whose coordinates is zero, we introduce a symbolic expression called multi-sign
\[
\mathrm{MS}(x) = (\pm, \pm, \pm)
\]
where the $i$-th sign  is that of the $i$-th component.   For example, the multi-sign $\mathrm{MS}(x)$ of a point $x$ in $\Delta_2$ is $(+, +, +)$, while $\mathrm{MS}(y)$ for a point $y$ in $\tilde{\Delta}_2$ is $(-, -, -)$.

We note that any two spherical simplices, which are also two geodesic triangles,  are projectively equivalent, since an arbitrary spherical simplex is projectively equivalent (by stereographic projection) to a Euclidean 2-simplex, that is, a triangle, and two arbitrary Euclidean triangles are projectively equivalent. 
In particular, the fact that two spherical simplices $D_1$ and $D_2$ are projectively equivalent implies that  the timelike Hilbert geometries of $S^2-(D_1\cup \tilde{D_1})$ and $S^2-(D_2\cup \tilde{D_2})$ are isometric, where for $i=1, 2$, $\tilde{D}_i$ is the antipodal simplex of $D_i$.

 Let $\mathbb{U}_3$ be a hemisphere of $S^2$ with pole $C_3=(0, 0, 1)$ and let $\Pi_{3} \subset \mathbb{R}^{3}$  be the hyperplane $\{x_3 =1 \}$ tangent to the northern hemisphere $\mathbb{U}_3$  at $C_3$ (See Figure 4a-1.) We also let $\tilde{\mathbb{U}}_3$ to be the southern hemisphere antipodal to $\mathbb{U}_3$. Its pole is the point $\tilde{C}_3 := (0, 0, -1)$. We let $\tilde{\Pi}_{3}$ be the hyperplane $\{x_1 =-1 \}$.
The plane $\Pi_{3}$ has a natural coordinate system $t^{(1)}= (x_2, x_3)$, and so does $\tilde{\Pi}_{3}$.  

The \emph{stereographic projection} $\pi_3$ associated with $\mathbb{U}_3$ centered at $C_3$ is the map  $\pi_3: \mathbb{U}_3 \to \Pi_{3}$  which sends each point $x$ of $\mathbb{U}_3$ to the intersection of the line $Ox$ with $\Pi_{3}$.  

$\mathbb{U}_1, C_1, \Pi_{1},\mathbb{U}_2, C_2, \Pi_{2}$ are defined accordingly.

Now note that the northern hemisphere $\mathbb{U}_3$ has,  on its equatorial boundary $\{x_1 =0\}$, four points $C_1, \tilde{C}_1, C_2, \tilde{C}_2$ where each of the $x_1$ and  $x_2$ coordinate axes  intersects the unit sphere.    The antipodal pair $C_i$ and    
$\tilde{C}_i \quad (i =1,2)$ determines a family of great circles $S^1$ passing through them, which gives a geodesic foliation of the open hemisphere $\mathbb{U}_1$. Each leaf $\ell$ of this foliation determines a $2$-plane $\Pi_{\ell}$ containing $\ell$ and the origin of $\mathbb{R}^3$.  The open semicircle $\ell$ is sent by the stereographic projection $\pi_3: \mathbb{U}_3 \rightarrow \Pi_3$ to a line  
$\mathcal{P}^{(i)}_{\ell}$ which is the intersection of the $2$-plane $\Pi_{\ell}$ with the $2$-plane $\Pi_{C_3} =\{x_3 =1 \}$.  (See Figure 4a-2.) 

The preceding observation says that given an arbitrary point $x$ in $\mathbb{U}_3$, there exists a uniquely determined pair of great circles $\ell^1(x),  \ell^2(x)$ intersecting  at $x$. (See Figure 4a-2.) The images of the great circles by the stereographic projection $\pi_3$ are respectively the pair of lines $\mathcal{P}^{(3)}_{\ell^1(x)}, \mathcal{P}^{(3)}_{\ell^2(x)}$  intersecting  at $\pi_3(x)$ perpendicularily, parallel to $x_1$ and $x_2$-axes in $\mathbb{R}^3$. (See Figure 7a.)   Hence at $\pi_3(x)$, $\mathcal{P}^{(3)}_{\ell^1(x)}, \mathcal{P}^{(3)}_{\ell^2(x)}$  constitute a double cone $J^\pm(\pi_3(x))$, obtained by translating the first and third orthants of $\Pi_{3}$.  We define the interior of  $J^+(\pi_3(x))$, consisting of points whose coordinates are all greater than those of $\pi_3(x)$, to be the future set of $\pi_3(x)$  and $J^-(\pi_3(x))$, the past set of $\pi_3(x)$. 

Note that the entire sphere is covered by the collection of charts 
\[
\mathbb{U}_1, \mathbb{U}_2 , \mathbb{U}_{3}, \tilde{\mathbb{U}}_1, \tilde{\mathbb{U}}_2 , \tilde{\mathbb{U}}_{3}.
\]
Each $\mathbb{U}_i$ contains the spherical simplex $\Delta_2$, the future convex set,  and each $\tilde{\mathbb{U}}_i$ contains the spherical simplex $\tilde{\Delta}_2$,  the past convex set. The image of the simplex  in $\pi_i(\mathbb{U}_i)$ is one of the four orthants  in the $3$-dimensional hyperplane $\{ x_i =1\}$, while $\tilde{\pi}_i(\tilde{\mathbb{U}}_i)$ is one of the four orthants  in the hyperplane $\{ x_i =-1\}$.  Any point $x$ in $\Omega = S^2 \setminus (\Delta_2 \cup \tilde{\Delta}_2)$ is covered by at least two local charts, $\mathbb{U}_i$ whose pole is a vertex of $\Delta$, and $\tilde{\mathbb{U}}_j$ whose pole is a vertex of $\tilde{\Delta}$.  For example, if $\mathrm{MS}(x) = (+, -, +)$, then $x$ lies in $\mathbb{U}_3$ as well as $\tilde{\mathbb{U}}_2$.    (See Figure 4c.) 
%\begin{figure}[!ht] 
%\centering
%\includegraphics[width=0.45\linewidth]{Figure_4a1.eps, Figure_4a2, Figure_4b.eps}    \caption{\small {local charts for timelike spherical Hilbert geometry}   \label{local charts}  
%\end{figure}

\begin{figure}
\label{fig4}
\centering
 \includegraphics[width=.9\linewidth]{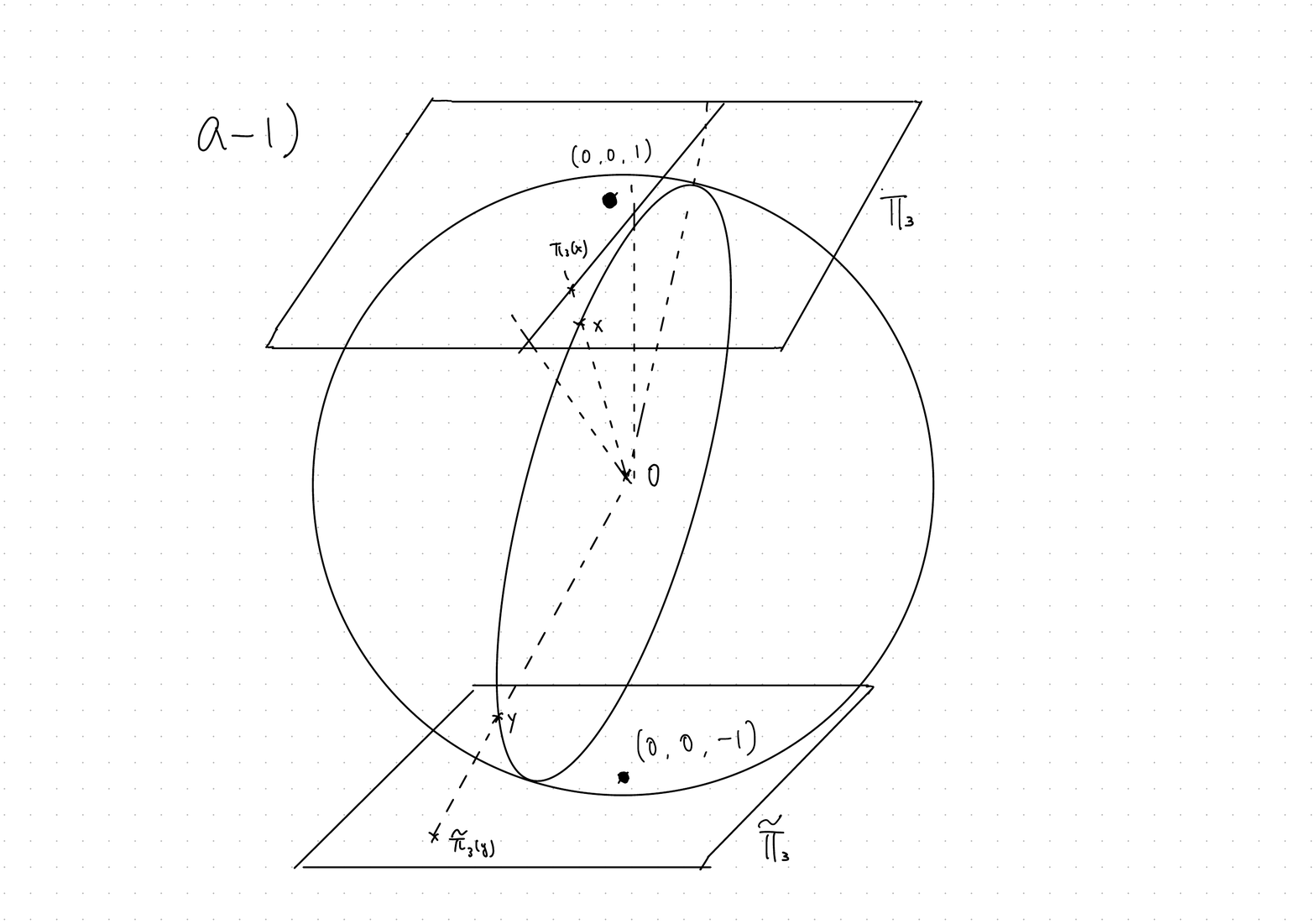}  
 \includegraphics[width=.8\linewidth]{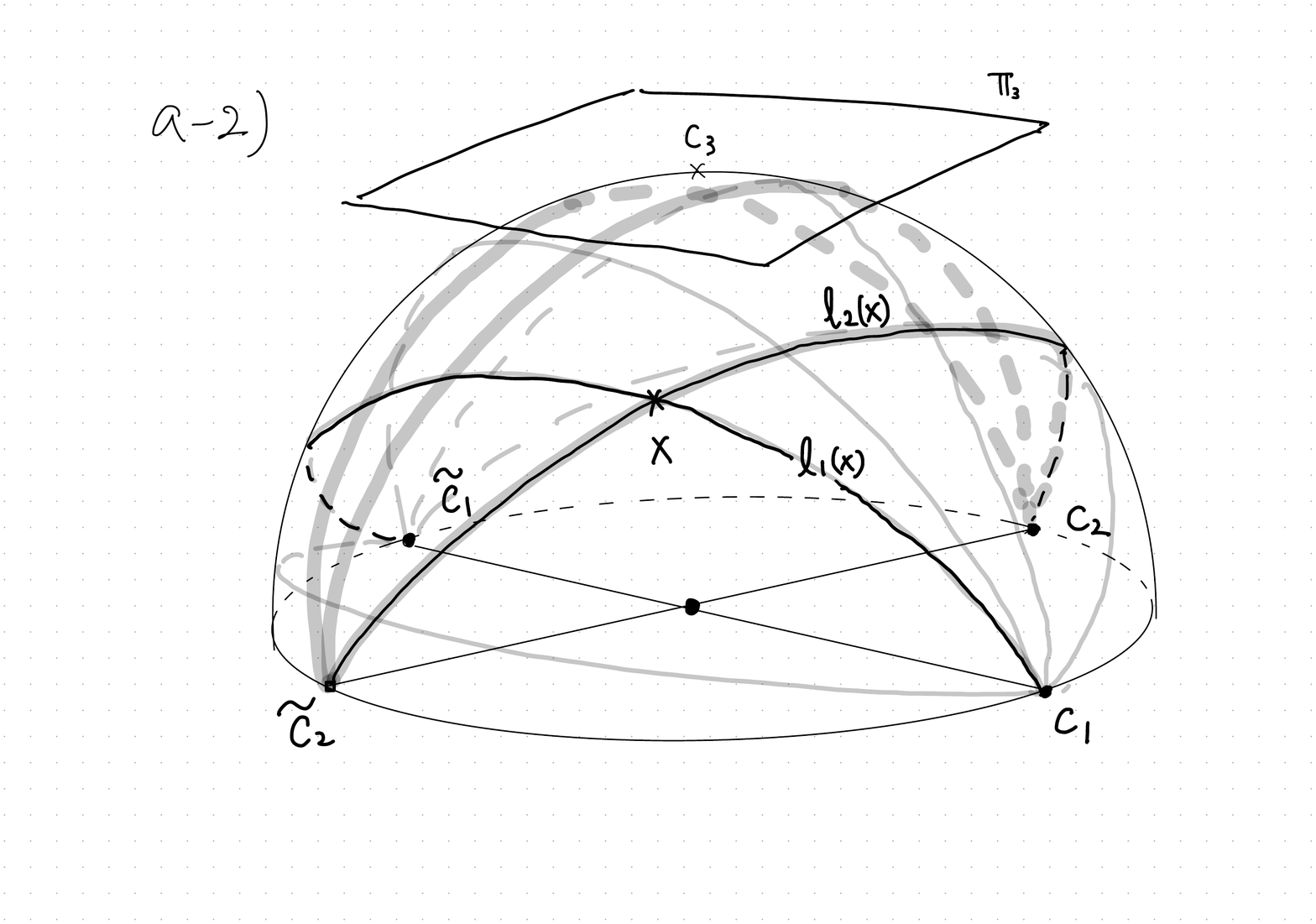}   
 \includegraphics[width=1\linewidth]{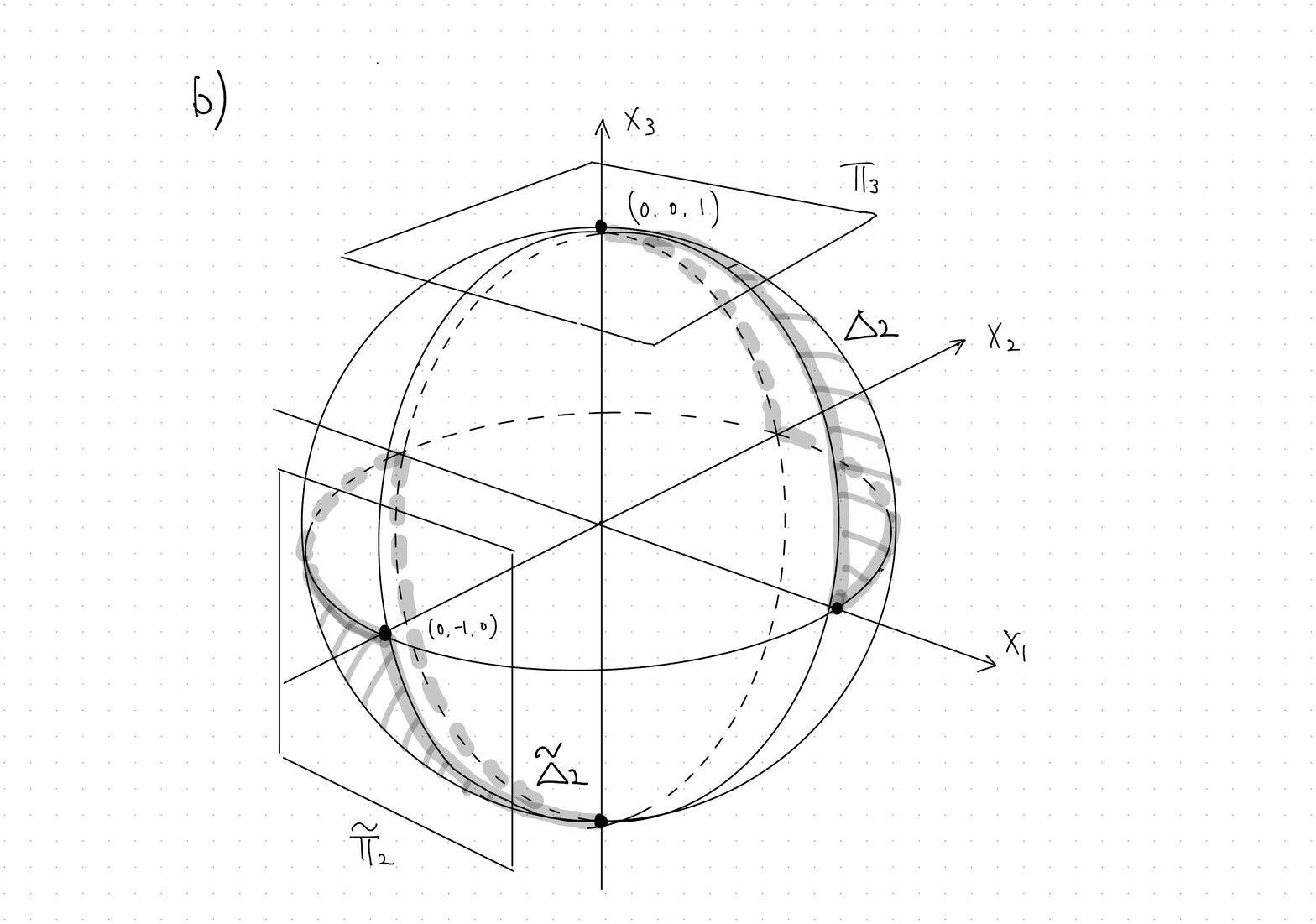}  
   \caption{\small {Local charts for timelike spherical Hilbert geometry}}
\end{figure}

     \section{The Hilbert metric of the $2$-simplex}\label{sec:Hilbert2}

Our main concern in this paper is the timelike Hilbert geometry of the 2-simplex, but we start by recalling the (non timelike) Hilbert geometry of the 2-dimensional simplex, that is, the triangle in the plane, for which explicit formulae are known. 
This is described by Phadke in his paper \cite{Phadke}. 
We start with this case because there are analogies between our approach to the  timelike Hilbert geometry of the 2-simplex and this simpler case of the (non-timelike) Hilbert geometry of the 2-simplex.

Any two triangles in the plane are projectively equivalent, and thus any triangle is projectively equivalent to the first quarter plane (or quadrant) of $\mathbb{R}^2$, $\Delta=\{(x_1,x_2)\in\mathbb{R}^2 \,| \, x_1>0, x_2>0\}.$ Therefore, the Hilbert metric of any triangle is isometric to the Hilbert metric of the quarter plane $\Delta$. Thus, it suffices to write the explicit formulae for the latter case. 
 
 The Hilbert distance $H_{\Delta}(p,q)$ between two points $p$ and $q$ in $\Delta$ is equal to half of the logarithm of the cross ratio of the ordered quadruple $a,p,q,b$ where $a$ and $b$ are the intersection points of the line $pq$ with the boundary of the quarter plane (they are either on the coordinate axes, or at infinity).
  The result depends on the relative position of the points $p$ and $q$ with respect to the two axes. There are 6 cases  (see Figure 5a):

\begin{itemize}

 \item  $H_{\Delta}(p,q)= \frac{1}{2}\bigl(\log (p_2/q_2)+ \log (q_1/p_1)\bigr)$ if $a$ is on the $y$ axis and $b$ on the $x$ axis;
 \item  $H_{\Delta}(p,q)= \frac{1}{2}\bigl(\log (q_2/p_2)+ \log (p_1/q_1)\bigr)$ if $a$ is on the $x$ axis and $b$ on the $y$ axis;
 \item $H_{\Delta}(p,q)= \frac{1}{2}\log (q_1/p_1)$ if $a$ is on the $y$ axis and $b$ at infinity;
\item $H_{\Delta}(p,q)=  \frac{1}{2}\log (p_1/q_1)$ if $a$ is at infinity and $b$ on the $y$ axis;
 \item $H_{\Delta}(p,q)= \frac{1}{2}\log (q_2/p_2)$ if $a$ is on the $x$ axis and $b$ at infinity;
 \item $H_{\Delta}(p,q)= \frac{1}{2}\log (p_2/q_2)$ if $a$ is at infinity and $b$ on the $x$ axis.
\end{itemize}

 The following formula combines all these cases: 

\begin{equation} 
H_{\Delta} (p,q)=\frac{1}{2}\max\biggl\{
\Bigl|\log \frac{p_1}{q_1}\Bigr|, 
\Bigl|\log \frac{p_2}{q_2}\Bigr|,
\Bigl|\log \frac{p_1 \, q_2}{q_1 \, p_2} \Bigr| \biggr\}.
\end{equation}
\medskip

Note that the  Euclidean 2-simplex is projectively equivalent to the spherical simplex in the unit 2-sphere, with local chart given by $(\mathbb{U}_3, \pi_3)$, in such a way that the spherical simplex $\{ x \in S^2 \, | \, \mathrm{MS}(x) = (+, + , +)\}$ is sent to the first quadrant $Q_1 \cong \Delta$.  Here, as before, projective equivalence means the existence of a homeomorphism which preserves the cross ratio of 4 aligned points, being understood that on the sphere we use the spherical cross ratio (see, for example \cite{2012-Hilbert3}).  In this sense, the Hilbert geometry of the simplex on the sphere can be regarded as a classical (non-timelike) spherical Hilbert geometry.  We remark that in the present situation, the antipodal simplex on $S^2$ is not apparent, yet it is still present in the projective geometry of $S^2$.   
  
%\begin{figure}[!ht] 
%\centering
%\includegraphics[width=0.55\linewidth]{Figure_5a.eps, Figure_5b.eps}    \caption{\small {Planar Funk models for Hilbert geometries}   \label{Planar Funk}  
%\end{figure}

\begin{figure}
\label{fig5}
\centering
 \includegraphics[width=1\linewidth]{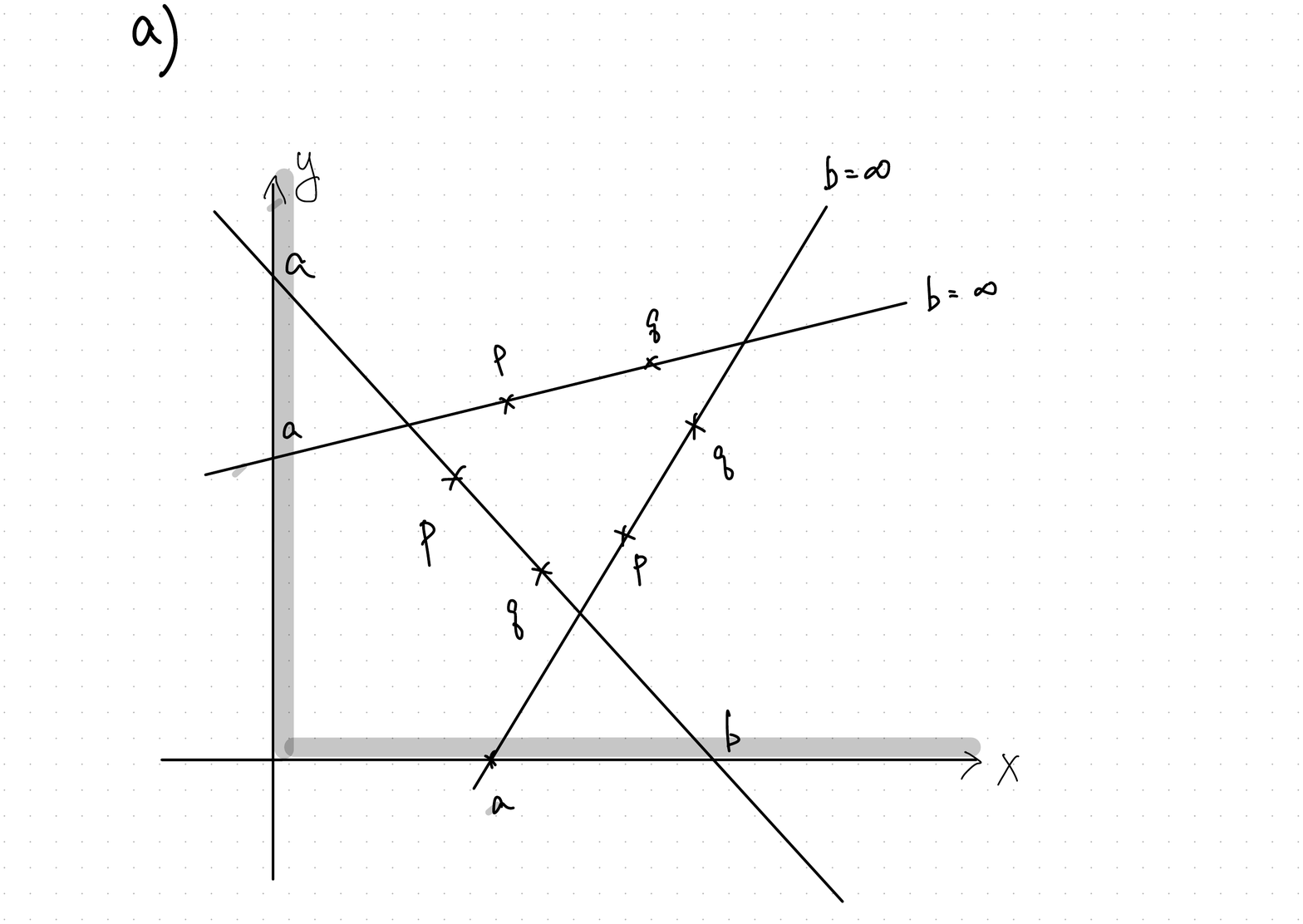}  
 \includegraphics[width=1\linewidth]{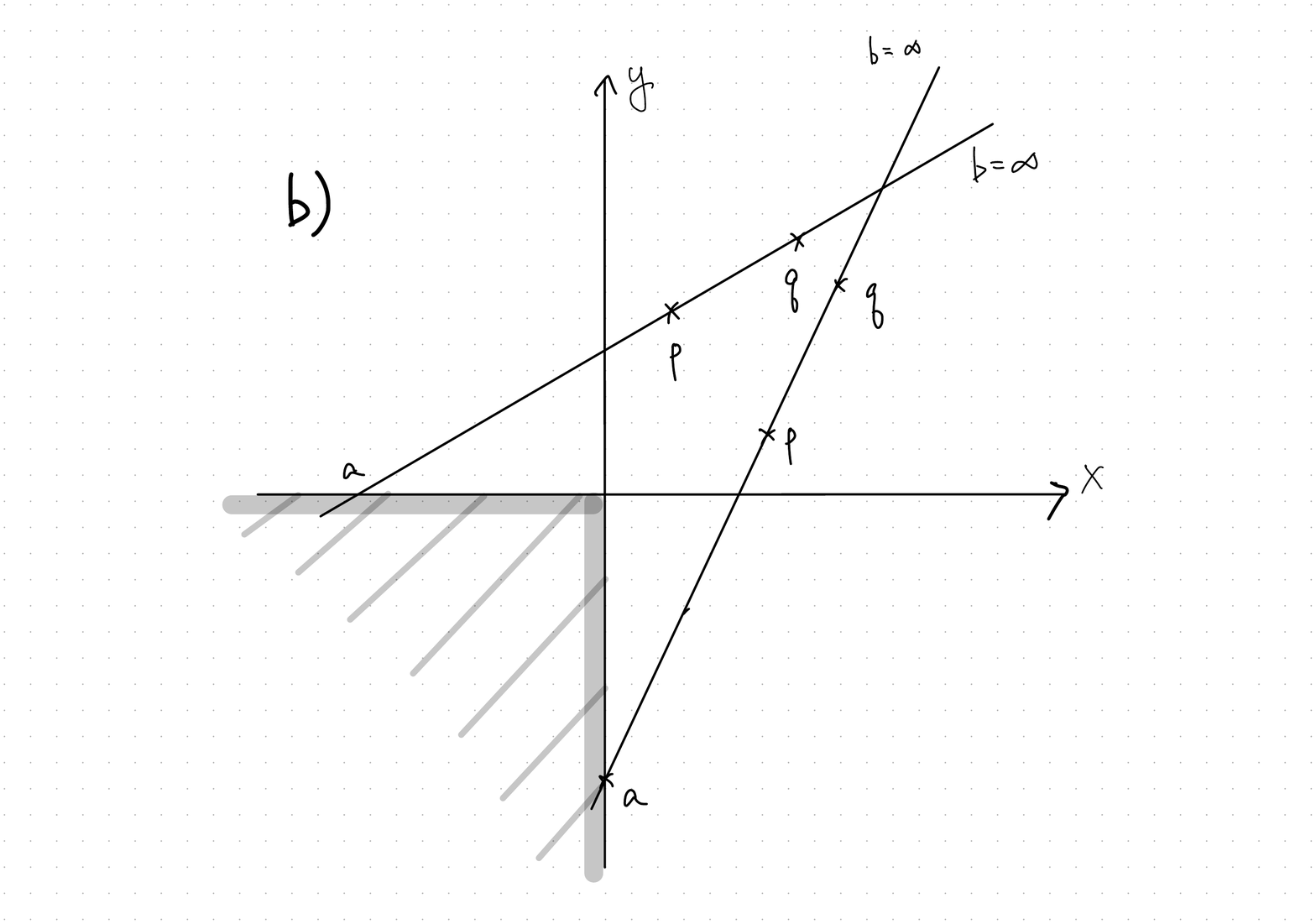}  
   \caption{\small {Planar Funk models for Hilbert geometries}}  
\end{figure}

\section{The timelike Hilbert geometry of the spherical $2$-simplex } \label{s:timelike-simplex}
 In this section, we prove the main theorem. We use the notation of \S \ref{s:decomposition}.
%We use the notation $K_1, K_2, \Omega_<$ and the spherical cross ratio introduced in \S \ref{s:timelike-Hilbert}.

 \begin{theorem}\label{th:1}
Let $\Delta_2$ be the standard $2$-dimensional simplex in the sphere $S^2$, with $\tilde{\Delta}_2$ its antipodal simplex. Then the timelike Hilbert geometry of $\Omega= S^2\setminus ( \Delta_2 \cup \tilde{\Delta}_2)$ is isometric to a union of six copies of normed spaces which are homeomorphic to $S^{1} \times \mathbb{R}$ on which the abelian group $(\mathbb{R}_{> 0})^2  \times \mathbb{Z}_{3} \times  \mathbb{Z}_{2} $ acts isometrically. 
  \end{theorem}

\begin{proof}[Proof of Theorem  \ref{th:1}]

Our configuration is to set the antipodal pair of two-dimensional spherical simplices $\Delta_2$ and $\tilde{\Delta}_2$ as the spherical cap in the first orthant and its antipodal counterpart respectively, namely,   
\[
\Delta_2 = \{ x \in S^2 \, | \, \mathrm{MS}(x) = (+, + , +)\}, \qquad \tilde{\Delta}_2 =  \{ x \in S^2 \, | \, \mathrm{MS}(x) = (-, - , -)\}.
\]

We first define precisely the stereographic projection. Note that in the literature this term is at times used differently.  

The stereographic projection $\tilde{\pi}_2$ sends points in the hemisphere $\tilde{\mathbb{U}}_2$ onto a plane $\tilde{\Pi}_2 := \{ x_2 = -1\}$ tangent to the sphere at a vertex $\tilde{C}_2$ of $\tilde{\Delta}_2$,  which we orient so that its normal vector is $(0, -1, 0)$, by the correspondence 
\[ 
x \mapsto Ox \cap \tilde{\Pi}_2.
\] 
This projection sends $\tilde{\Delta}_2$ to a quadrant in $\tilde{\Pi}_2$, which is the third quadrant $Q_3 := \{x_1<0, x_3< 0\}$ in the plane $(x_1,x_3)$. By this projection, our timelike spherical Hilbert metric is isomorphic to the (degenerate)  timelike  Hilbert metric of the plane $\mathbb{R}^2 \setminus Q_3$ in which the past  is the third quadrant $Q_3$ and the future consists of the points at infinity of the first quadrant $\{``r = \infty", 0 < \theta < \pi/2 \}$. It is degenerate in the sense that one of the two convex sets is infinitely far away (see Figure 5b), making the Hilbert metric effectively a Funk metric on $\tilde{\Pi}_2$ (see \S \ref{s:timelike-Hilbert}).  For the sake of simplicity, we use the coordinates $(x, y)$ in place of $ t^{(2)} =(x_1, x_3)$.

We recall from \S \ref{s:timelike-Hilbert}
 that the timelike spherical Hilbert distance $H_{\Delta}(p,q)$ of two points $p$ and $q$ with the partial order $p< q$ in $\Omega$ is equal to half of the spherical cross ratio of an ordered quadruple $(a_1, p, q, a_2)$ aligned on a great circle $\ell$ on $S^2$, in that order. These four points are sent to  $\tilde{\pi}_2(a_1),\tilde{\pi}_2(p),\tilde{\pi}_2(q),\tilde{\pi}_2(a_2)$ in $\tilde{\Pi}_2$, where $\tilde{\pi}_2(a_1) =:a$ is the intersection point of the line $\tilde{\pi}_2(\ell)$ with the boundary of the third quadrant $Q_3 := \{ x< 0, y< 0\}$ and $\tilde{\pi}_2(a_2) =:b$  is at the infinity of the first quadrant.  In coordinates, we write $\tilde{\pi}_2(p) = (p_1, p_2)$ and $\tilde{\pi}_2(q)=(q_1, q_2)$ in the plane $\tilde{\Pi}_{2}$.  The partial ordering $p< q$ is equivalent to the line through $p$ and $q$ having a positive slope. In other words, we have $p_1 \leq q_1$ and $p_2 \leq q_2$.

There are two cases: 

\begin{itemize}
\item $H_{\Delta}(p,q)= \frac{1}{2}\log (q_2/p_2)$ if $a$ is on the negative part of the $x$ axis and  $b$  at infinity;
\item $H_{\Delta}(p,q)= \frac{1}{2}\log (q_1/p_1)$  if $a$ is on the negative part of the $y$ axis and  $b$ at infinity.
\end{itemize}

 The following formula combines the two cases: 

\begin{equation} 
H_{\Delta} (p,q)=\frac{1}{2}\max\biggl\{
\Bigl|\log \frac{p_1}{q_1}\Bigr|, 
\Bigl|\log \frac{p_2}{q_2}\Bigr|
 \biggr\}.
\end{equation}
\medskip

We note now that, from this expression of the Hilbert distance $H_{\Delta}$, the affine map defined on $\tilde{\Pi}_2$ by
\[
x \mapsto \lambda_1 x, \quad y \mapsto \lambda_2 y \qquad (\lambda_i >0)
\]
is distance-preserving.  Hence the multiplicative abelian group $\{(\lambda_1, \lambda_2)\} \cong (\mathbb{R}_{>0})^2$ is a subgroup of the isometry group of the Hilbert metric.

Next,  we shall give formulae for the Finsler structure associated with this  timelike spherical Hilbert metric. We shall show that the timelike metric induced by the timelike Hilbert metric on this quadrant is homogeneous. More precisely, it is isometric to  a union of six copies of timelike normed spaces.  

Since we are giving formulae for the local Finsler structure, there is no loss of generality in restricting our study to the first quadrant in the plane $(x, y)$, which in our context is $\tilde{\Pi}_2$. The Finsler structure on the five other quadrants can be obtained from this special one by the $(\mathbb{Z}_2\times \mathbb{Z}_3)$-symmetry, where  $\mathbb{Z}_2$ is the symmetry generated by the antipodal map of the ambient space 
\[
(x_1, x_2, x_3) \mapsto (-x_1, -x_2, -x_3) 
\] 
and $\mathbb{Z}_3$ is the rotational symmetry group generated by the permutation of the coordinates
\[
(x_1, x_2, x_3) \mapsto (x_2, x_3, x_1).
\]

Let $\mathbf{v} = (v_1,v_2)$ be the coordinates of a vector at a point $\mathbf{x} = (x, y)$ in the first quadrant $Q_1 := \{ (x> 0, y>0 \}$.  We recall the general formula \cite{PY} for the Minkowski functional for the Hilbert metric modeled on the plane. Given a point x and a tangent vector $v$ at $\mathbf{x}$, let $r^+ = r^+(\mathbf{x}, \mathbf{v})$ be the distance between $\mathbf{x}$ and the point where the ray $\mathbf{x}+ t\mathbf{v}  \quad (t>0)$ intersects the boundary of the convex set, and $r^-=r^-(\mathbf{x}, \mathbf{v})$ the distance between $x$ and the point where the ray $\mathbf{x} -t\mathbf{v} \quad  (t<0)$  intersects the boundary of the convex set.  Then the Minkowski functional $p(\mathbf{x}, \mathbf{v})$ for the Hilbert metric $H(\mathbf{x}, \mathbf{y})$ is given by 
\[
p(\mathbf{x}, \mathbf{v}) = \frac12 |v| \Big( \frac{1}{r^+} + \frac{1}{r^-} \Big).
\]
Clearly this expression is symmetric in $\mathbf{v}$; $p(\mathbf{x}, \mathbf{v})= p(\mathbf{x}, -\mathbf{v})$. Furthermore,  the expression reflects the fact that the Hilbert metric is the arithmetic symmetrization of the timelike Funk metrics, as each term is the Minkowski functional of the respective  timelike Funk metric for the future convex set $\Delta_2$ and the past convex set $\tilde{\Delta}_2$.

Also note that the line segment with positive slopes in $\mathbb{R}^2 \setminus Q_3$ corresponds to great circles in the sphere by the inverse of the stereographic projection $\tilde{\pi}_2: \tilde{\mathbb{U}}_2 \rightarrow \tilde{\Pi}_2$. 

 Given a point $p$ in $\mathbb{R}^2 \setminus Q_3$, let $x$ be the point $\tilde{\pi}_2^{-1}(p)$ in $\tilde{\mathbb{U}}_2$.  The inverse images of all the line segments of positive slope through $p$ are then identified with the region defined by two great semi-circles $\ell_1(x)$ and $\ell_2(x)$, the former through $C_1:= (1, 0, 0)$ and $\tilde{C}_1 := (-1, 0, 0)$, and the latter through $C_3:= (0, 0, 1)$ and $\tilde{C}_3 := (0, 0, -1)$.  Note that the spherical region, which is a union of two spherical lunes, contains the two simplices $\Delta_2$ and $\tilde{\Delta}_2$, marking the future and the past of the point $x$ respectively. Also note that $C_2:= (0, 1, 0)$ cannot be reached by any arc of great circle from $p$ without traversing one of the simplices.  

The great circles $\ell_1(\mathbf{x})$ and $\ell_2(\mathbf{x})$ are sent by $\tilde{\pi}_2$ to two straight lines in $ \tilde{\Pi}_2 \cong \mathbb{R}^2$ that are the horizontal (parallel to the $x$-axis) and vertical (parallel to the $y$-axis) lines intersecting at $p$, hence they do not intersect the third quadrant (See Figure 7a), implying that
\[
r^+(\mathbf{x}, \mathbf{v}) = r^-(\mathbf{x}, \mathbf{v}) = \infty,
\]
making 
\[
p(\mathbf{x}, \mathbf{v}) = \frac12 |v| \Big( \cancel{\frac{1}{r^+}} + \cancel{\frac{1}{r^-}} \Big) = 0
\]
when $\mathbf{v}$ is tangent to the great circles $\ell_1(\mathbf{x})$ or $\ell_2(\mathbf{x})$.
In other words, these lines define the lightlike/null directions of the timelike spherical Hilbert metric. (See Figure 6a.) This is in analogy with the case of the classical Lorentz space\index{Lorentz space} where the light cone is the set of vectors of norm zero ($\lambda_n(\mathbf{x})=0$ in the notation of \S \ref{s:Intro}).

%\begin{figure}[!ht] 
%\centering
%\includegraphics[width=0.55\linewidth]{Figure_6a.eps, Figure_6b.eps}    \caption{\small {Light cone structures}   \label{light cone}  
%\end{figure}

\begin{figure}
\label{fig6}
\centering
 \includegraphics[width=1\linewidth]{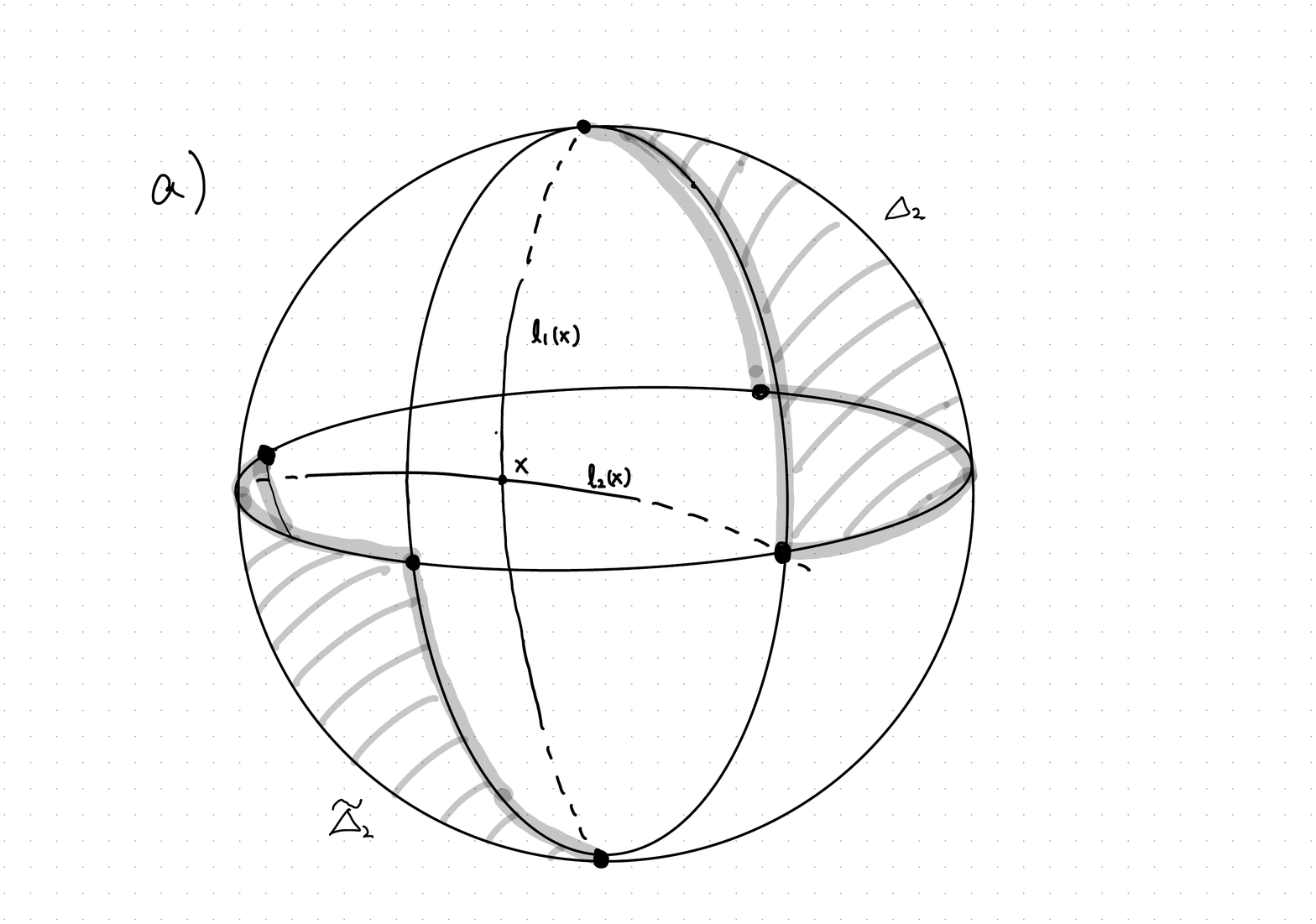}  
 \includegraphics[width=1\linewidth]{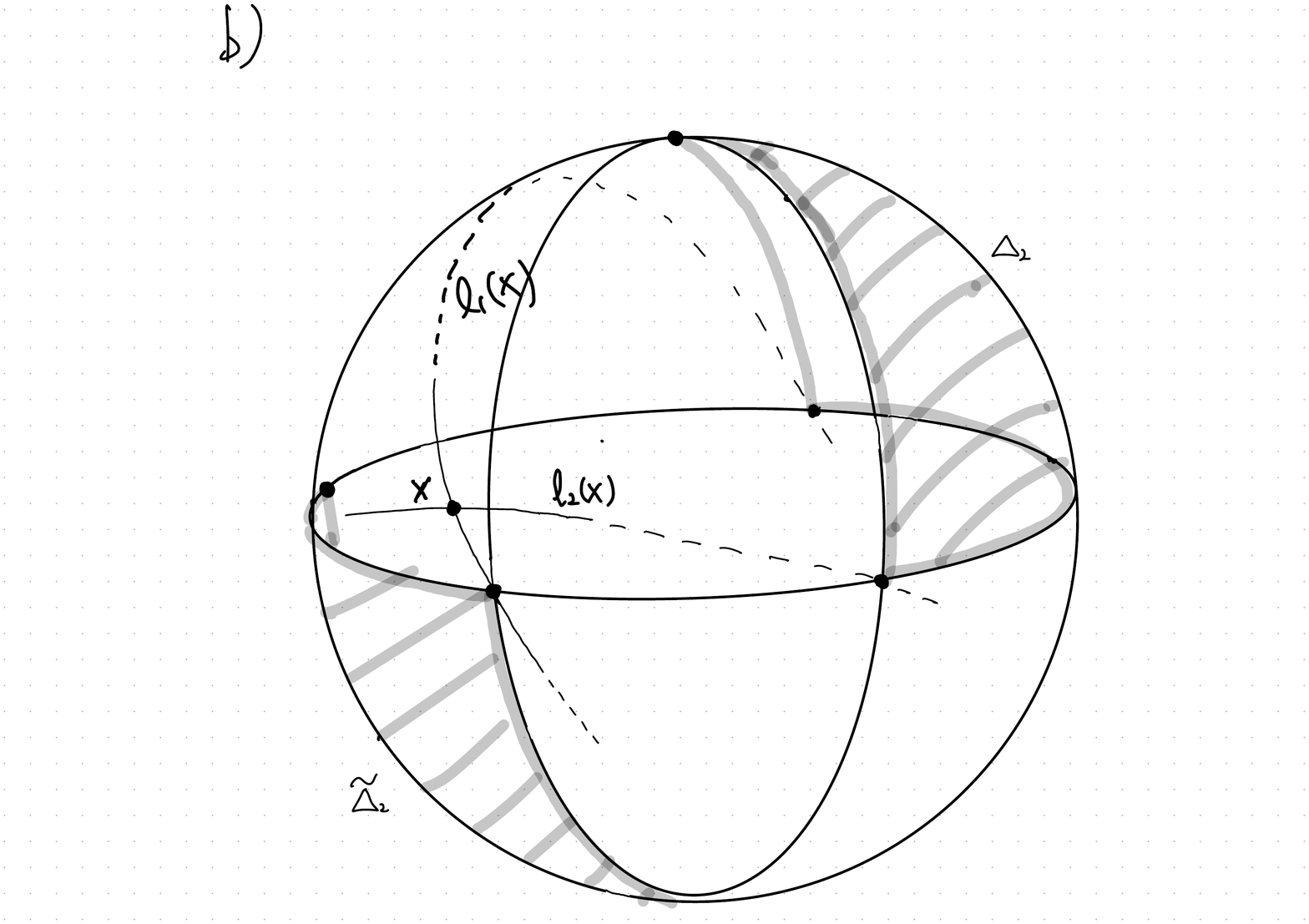}  
   \caption{\small {Light-cone structures}}  
\end{figure}

The explicit expression of the Minkowski functional for  the vector is nonzero only if the ray along the vector hits the third quadrant $Q_3$.
There are two cases for the value of the Minkowski functional on that vector, depending on whether this ray hits the negative part  of the $x$ axis or the negative part  of the $y$ axis (see Figure 7a).  

Case 1 is characterized by the inequality $\frac{v_2}{v_1}<\frac{y}{x}$. Letting $\theta$ be the angle made by the vector with the negative $y$-axis, and $r_+$ the distance between the point of coordinates  $(x, y)$ and the point where the ray directed by the vector $(v_1,v_2)$ hits the $x$-axis, we have, in this case, $r_+=\frac{y}{\cos \theta}$ and 
\[
p(\mathbf{x},v)=
\frac12 \vert v \vert \left(\frac{1}{\frac{y}{\cos \theta}} \right)
= \frac12 \frac{\vert v\vert \cos \theta}{y}=- \frac12 \frac{v_2}{y}>0
.\]

 Case 2 is characterized by the inequality $\frac{v_2}{v_1}>\frac{y}{x}$. Letting $\sigma$ be now the angle made by the vector with the negative $x$-axis,   and $r_+$ the distance between the point of coordinates  $(x, y)$ and the point where the ray directed by the vector $(v_1,v_2)$ hits the $x$-axis, we have $r_+=\frac{x}{\cos \sigma}$ as before, and 
\[
p(\mathbf{x},v)= \frac12
\vert v\vert \left(\frac{1}{\frac{x}{ \cos \sigma}} \right)
= \frac12 \frac{\vert v\vert\cos\sigma}{x}=- \frac12 \frac{v_1}{x}>0
.\]

The intermediate case when $\frac{v_2}{v_1} = \frac{y}{x}$ is when the vector $v$ is directed toward the origin with
\[
p(\mathbf{x},v)=- \frac12 \frac{v_1}{x} = - \frac12 \frac{v_2}{y} > 0.
\] 

It is useful to change variables
\[
 \tilde{\mathbf{x}} := (\tilde{x}, \tilde{y}) 
 := (\log x, \log y) := \mathrm{Log}\, \mathbf{x} 
 \]
to describe the Minkowski norm $p$.  The logarithm function $\mathrm{Log}$ defines a diffeomorphism from $Q_1$ to $\mathbb{R}^2$. 
Such a transformation sends the first quadrant of the $(x,y)$-plane onto the whole  $(\tilde{x}, \tilde{y})$-plane, the latter being seen as the tangent space at a point in the former. 

The change of variables for tangent vectors is described by the Jacobian matrix:

\[
d \mathrm{Log}: 
\Big(\begin{matrix} v_1 \\  v_2\\  \end{matrix}\Big)\mapsto
\Bigg(\begin{matrix} \frac{1}{x} & 0 \\  0 &  \frac{1}{y}\  \end{matrix} \Bigg)
\Big(\begin{matrix} v_1 \\  v_2\\  \end{matrix} \Big)=
\Bigg(\begin{matrix} \frac{v_1}{x} \\  \frac{v_2}{y} \\  \end{matrix}  \Bigg) =:
\Big(\begin{matrix} \tilde{v}_1 \\  \tilde{v}_2\\  \end{matrix}\Big)
\]

Thus, the Minkowski functional, written in the  $(\tilde{x}_1, \tilde{x}_2)$-coordinates, is given by the formulae
\[
p(\mathbf{x},\mathbf{v})=-\frac12 \frac{v_2}{x_2}= - \frac12 \tilde{v}_2 := \tilde{p} (\tilde{\mathbf{x}},\tilde{\mathbf{v}})
\]
 in the case  
$\tilde{v}_1>\tilde{v}_2$
and
\[
p(\mathbf{x},\mathbf{v})=-  \frac12\frac{v_1}{x_1}=-\frac12 \tilde{v}_1>0 =:\tilde{p} (\tilde{\mathbf{x}},\tilde{\mathbf{v}})
\] 
in the case 
$\tilde{v}_1<\tilde{v}_2$.

These formulae,
\begin{align}
\tilde{p}(\tilde{\mathbf{x}}, \tilde{\mathbf{v}}) =  \begin{cases}
      - \frac12 \tilde{v_2}>0 & \text{if } \tilde{v}_1>\tilde{v}_2,\\
     - \frac12 \tilde{v_1}>0 & \text{if } \tilde{v}_1<\tilde{v}_2.
  \end{cases}
\end{align}
are independent of the $(\tilde{x}, \tilde{y})$-coordinates. Hence, the Finsler space is a Minkowski space, that is, a finite-dimensional normed space, and in particular the light cone structure of the metric is homogeneous in $(\tilde{x}, \tilde{y})$ (see Figure 8a).

\end{proof}

%\begin{figure}[!ht] 
%\centering
%\includegraphics[width=0.55\linewidth]{Figure_7a.eps, Figure_7b.eps}    \caption{\small {dependence on local charts}   \label{local charts}  
%\end{figure}

\begin{figure}
\label{fig7}
\centering
 \includegraphics[width=1\linewidth]{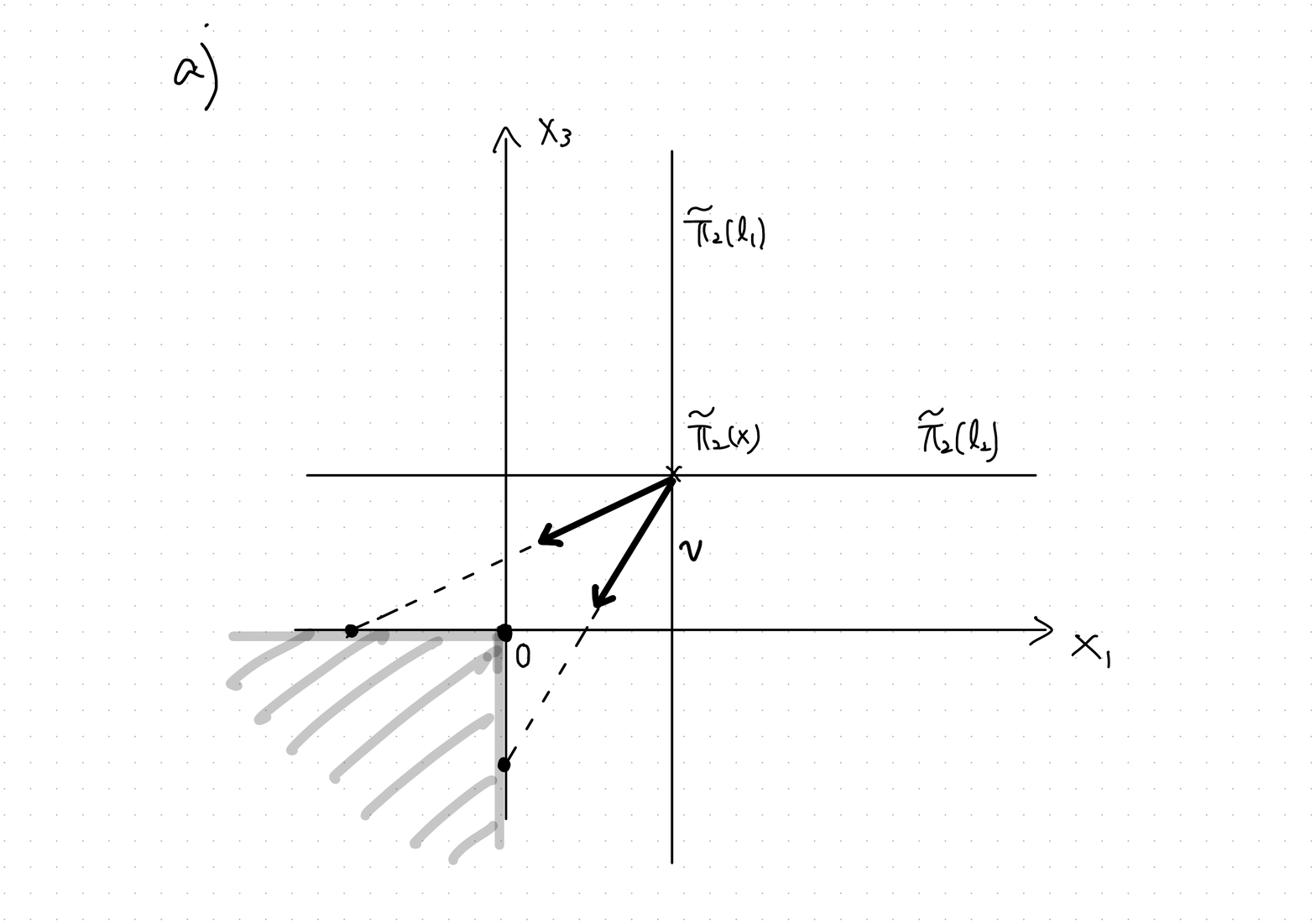}  
 \includegraphics[width=1\linewidth]{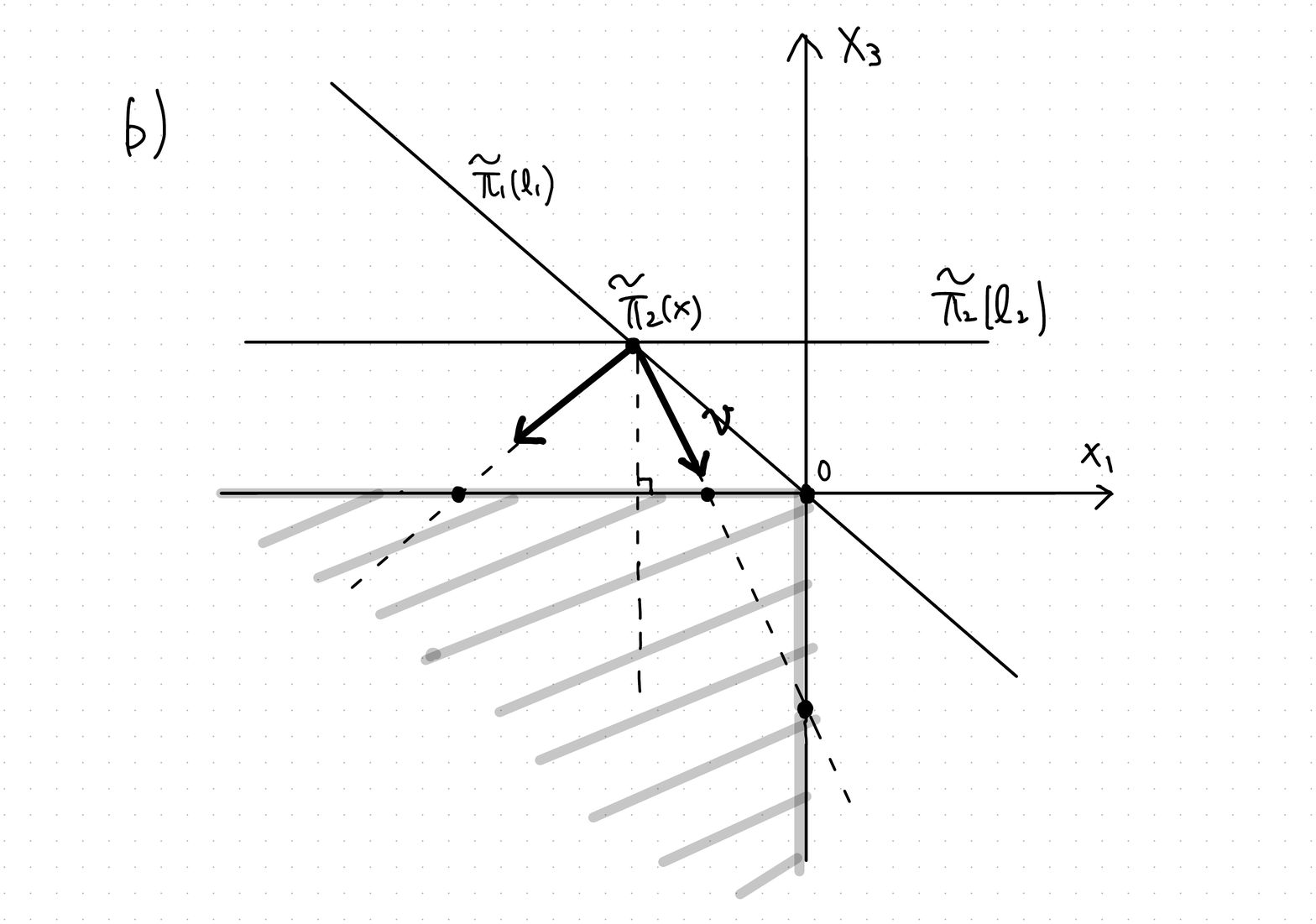}  
   \caption{\small {Local charts}}  
\end{figure}

Let us now study the Finsler structure at a point $\mathbf{x} = (x, y)$ in the second quadrant $Q_2 = \{ x< 0, y>0\}$. (See Figure 6b.) Recall that it is not necessary to prove the main theorem, as one can always change the local charts so that the above argument can be applied exactly in the new coordinate system.  Indeed, using the chart $(\mathbb{U}_3, \pi_3)$, the transformation 
$\pi_3 \circ (\tilde{\pi}_2)^{-1}$ will have $Q_2$ moved to the third quadrant $\overline{Q}_3$ of $\Pi_3$, the $x_2x_3$-plane, and the future set $\Delta_2$ is the first quadrant $\overline{Q}_1$ of $\Pi_3$.  Yet, we want to understand the effect of the choice of coordinates in terms of the representation of the Minkowski functional.
  
Let $\mathbf{v} = (v_1,v_2)$ denote the coordinates of a vector at that point $\mathbf{x}$. We are interested in the case when the ray from $\mathbf{x} + t\mathbf{v}$  hits the third quadrant $Q_3$. (See Figure 7b.)  The ray will then have to hit the negative part of the $x$-axis, with $v_2<0$ required.
There are, however, two cases to be distinguished from each other:

Case 1: $v_1\leq 0$;

 Case 2: $v_1 >0$. 
 
   Case 1 is simple, as once the ray hits $Q_3$, it remains in $Q_3$.  In this case, the Minkowski functional is simply
\[
p(\mathbf{x}, \mathbf{v}) = - \frac12  \frac{v_2}{y}. 
\]

In Case 2, the ray hits the quadrant $Q_3$ at $a_1$ and then leaves the quadrant  at a point $\hat{a}_2$ belonging to the negative part of the $y$-axis.  Indeed,  the Minkowski functional as the linearization of 
\[
H(\mathbf{x}, \mathbf{x} + t\mathbf{v}) = \frac12 \log [\hat{a}_2, \mathbf{x}, \mathbf{x} + t\mathbf{v}, a_1]
\]
at $t=0$ is calculated as
\[
p(\mathbf{x}, \mathbf{v}) = \frac{d}{dt} H(\mathbf{x}, \mathbf{x} + t\mathbf{v}) \Big|_{\{t=0\}} = \frac12 |v|  \Big( \frac{1}{r^+} - \frac{1}{\hat{r}^-} \Big),
\]
where $\hat{r}^-$ is the Euclidean distance from $\mathbf{x}$ to $\hat{a}_2$. 

We map $Q_2 = \{ x< 0, y>0 \}$ onto the whole plane $\mathbb{R}^2$ by the mapping
\[
\mathrm{Log}: (x, y) \mapsto (\log (-x), \log y)
\]
whose differential is 
\[
d \mathrm{Log}: 
\Big(\begin{matrix} v_1 \\  v_2\\  \end{matrix}\Big)\mapsto
\Bigg(\begin{matrix} \frac{1}{-x} & 0 \\  0 &  \frac{1}{y}\  \end{matrix} \Bigg)
\Big(\begin{matrix} v_1 \\  v_2\\  \end{matrix} \Big)=
\Bigg(\begin{matrix} - \frac{v_1}{x} \\  \frac{v_2}{y} \\  \end{matrix}  \Bigg) =:
\Big(\begin{matrix} -\tilde{v}_1 \\  \tilde{v}_2\\  \end{matrix}\Big)
\]
where $v_2 < 0$ which in turn implies $\tilde{v}_2 < 0$.   This change of variables induces the following expression for the Minkowski functional:

In  Case 1,  we have $v_1\leq 0$, and $r^-=\infty$ and hence
\[
p(\mathbf{x}, \mathbf{v}) = \frac12  |v| \Big( \frac{1}{r^+} +\cancel{\frac{1}{r^-}} \Big) = -\frac12 \frac{v_2}{y} = - \frac12 \tilde{v_2} > 0.
\]

In Case 2, we have $v_1 >0$, $\hat{r}^-= -\frac{x}{\sin \theta}$ where $\theta$ is the angle between $v$ and the $y$-axis. Hence 
\[
p(\mathbf{x}, \mathbf{v}) = \frac12  |v| \Big( \frac{1}{r^+} - \frac{1}{\hat{r}^-} \Big) = - \frac12 \frac{v_2}{y}   - \frac12 \frac{|v| \sin \theta}{x} = 
- \frac12 \frac{v_2}{y}  - \frac12 \frac{v_1}{x} = - \frac12 \tilde{v}_2 - \frac12 \tilde{v}_1.
\]

Note that the norm in Case 2 is positive, as  $\hat{r}^- > r^+$ is implied by the geometry. 

Thus, we have
\begin{align}
\tilde{p}(\tilde{\mathbf{x}}, \tilde{\mathbf{v}}) =  \begin{cases}
      - \frac12 \tilde{v_2} & \text{if } \tilde{v}_1 \leq 0,\\
     - \frac12 \tilde{v}_1- \frac12 \tilde{v_2}   & \text{if } \tilde{v}_1 > 0
  \end{cases}
\end{align}

Note that the Minkowski norm $\tilde{p}$ thus defined on $\mathbb{R}^2$ is independent of $(\tilde{x}, \tilde{y})$, making the timelike metric space homogeneous. 
The set of unit length vectors for this norm, that is, its indicatrix, as well as the light cones in this normed space, are represented in Figure 8b.  One can check that   for the corresponding indicatrix of the Minkowski functionals, Figure 8a and  Figure 8b  are identified via the transformation 
$\pi_3 \circ (\tilde{\pi}_2)^{-1}$. 

%\begin{figure}[!ht] 
%\centering
%\includegraphics[width=0.55\linewidth]{Figure_8a.eps, Figure_8b.eps}    \caption{\small {Minkowski functionals for normed space models}   \label{Minkowski functional}  
%\end{figure}

\begin{figure}
\label{fig8}
\centering
 \includegraphics[width=1\linewidth]{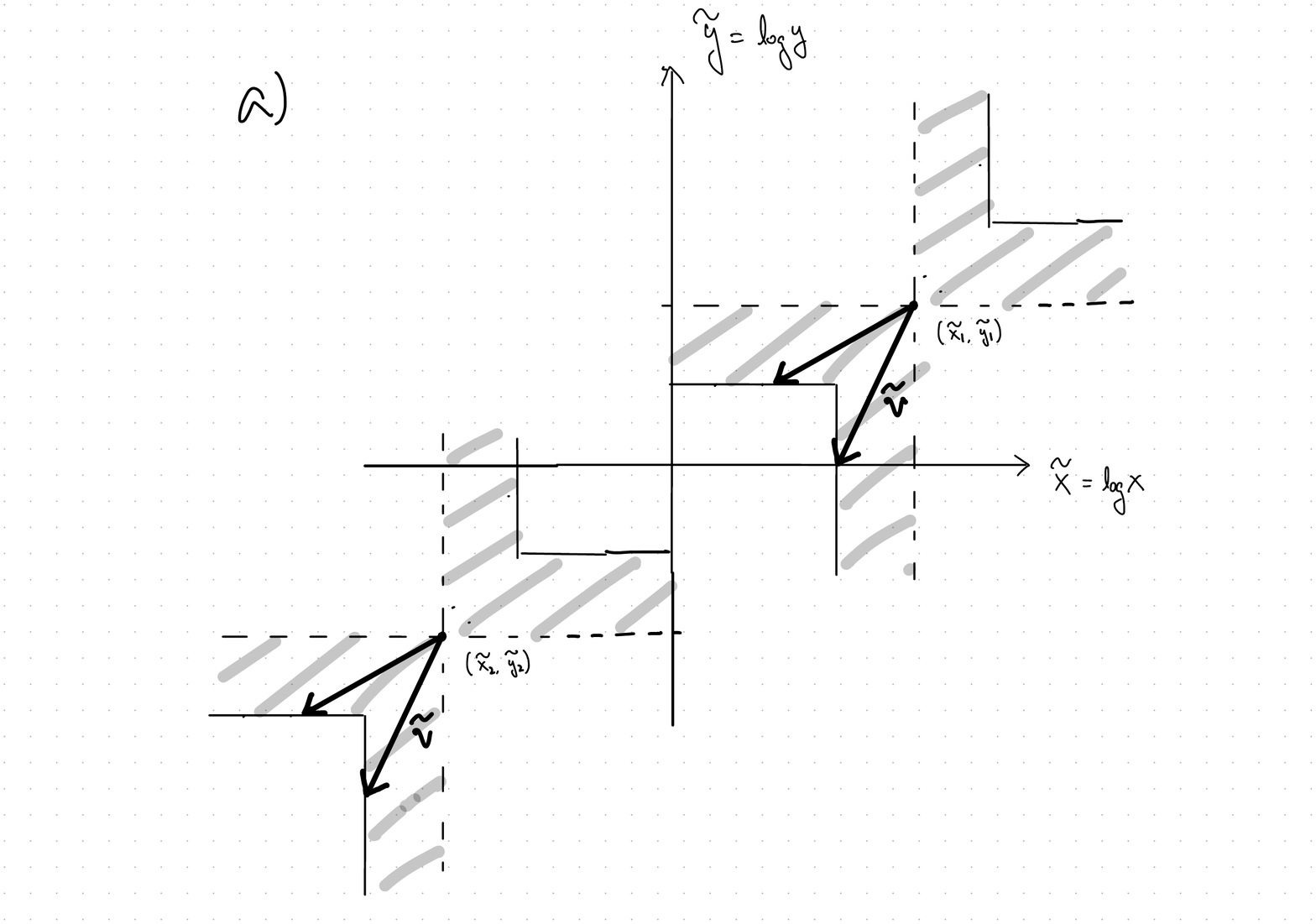}  
 \includegraphics[width=1\linewidth]{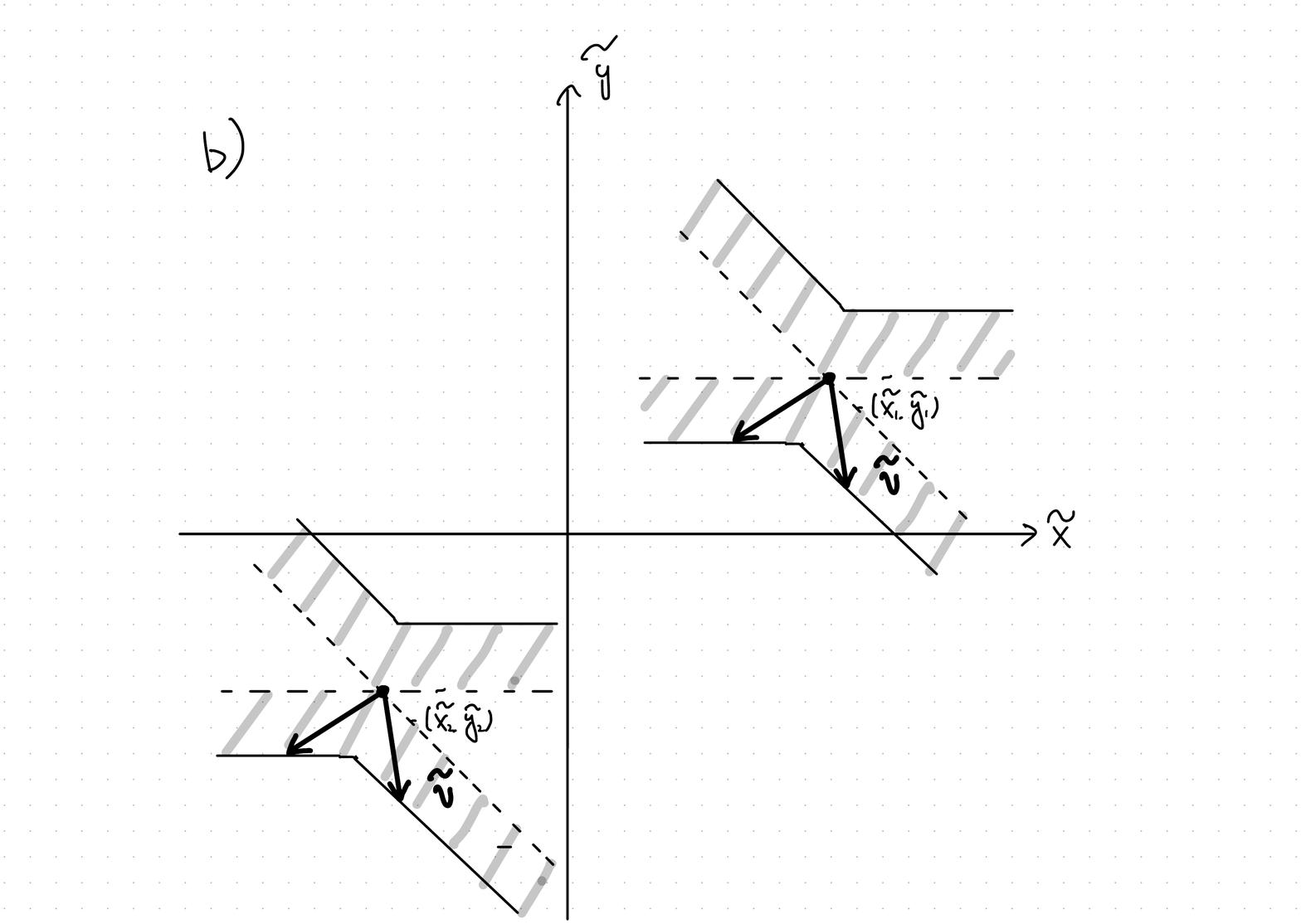}  
   \caption{\small {Minkowski functional}}  
\end{figure}

Hence, the timelike spherical Hilbert geometry of $\Omega = S^2 \setminus (\Delta_2 \cup \tilde{\Delta}_2)$ is homogeneous. The local homogeneity is  represented by $(\mathbb{R}_{> 0})^2$ as the structure of a timelike normed space, and the global homogeneity is given by the action of the abelian Lie group $(\mathbb{R}_{> 0})^2  \times \mathbb{Z}_{3} \times  \mathbb{Z}_{2} $.    
This local homogeneity can be  expressed  differently  depending on where the Minkowski norm is calculated.  

When  the local chart is chosen as above so that the basepoint $\mathbf{x}$ is in the first quadrant of $Q_1$, and the simplex $\tilde{\Delta}_2$ is $Q_3$, then the multi-sign $\mathrm{MS}(\mathbf{x}) = (+, - +)$ while $\mathrm{MS}(\mathbf{y}) = (-,-,-)$ for any point $\mathbf{y}$ in $\tilde{\Delta}_2$.  Going from $x$ to $y$,  the multi-sign $\mathrm{MS}$ has to change from $(+, -, +)$ to $(-,-,-)$. Namely, from the viewpoint of $\mathbf{x}$, there are two coordinate hyperplanes in $\mathbb{R}^3$ to cross:
\[
(+, -, +) \rightarrow (-, -, +) \rightarrow (-, -, -) \quad \mbox{ or } \quad   (+, -, +) \rightarrow (+, -, -) \rightarrow (-, -, -)
\]
 as one traverses the unit sphere, in order to reach the simplex $\tilde{\Delta}_2$ from $\mathbf{x}$. Each hyperplane corresponds to a side of the light cone at $\mathbf{x}$ in the normed space model, that is, rays parallel to the $\tilde{x}$ and $\tilde{y}$-axis, as the light cone is asymptotic to the direction tangential to the boundary of the convex set $\tilde{\Delta}_2$. 

On the other hand, the local chart is chosen as above so that the basepoint $\mathbf{x}$ is in the first quadrant of $Q_2$, and the simplex is $Q_3$, then as one traverses from $\mathbf{x}$ to a point $\mathbf{y}$ in $Q_3$, the multi-sign has to change from $(-, -, +)$ to $(-, -, -)$. 
Note here that  there is only one coordinate hyperplane, namely $(x_1,x_2)$-plane, to cross in order to reach $Q_3$ from $\mathbf{x}$. 
The past-directed light cone at $\mathbf{x}$ in the normed space model consists of two rays, one asymptotic to the $x$-axis, and the other passing through the origin of the $\tilde{x}\tilde{y}$-plane.  

 Of course, if one looks at the ``future" of $\mathbf{x}$, then the multi-sign changes from $(-,-, +)$ to $(+, +, +)$, and the situation is back to the first case of the multi-sign changes from $(+, -, +)$ to $(-,-,-)$, modulo the $\mathbb{Z}_2$ symmetry ``future" $\leftrightarrow$ ``past".   

It follows in particular that the shape of the light cone in the normed space representing the Hilbert geometry is dependent on the multi-sign of the basepoint $\mathbf{x}$, though they are all equivalent. (Compere Figure 8a and  Figure 8b.)

We finally note that Phadke, in the paper \cite{Phadke}, gave formulae for the classical (non-timelike) Minkowski functional associated to the interior of a simplex.
%, which was further investigated in \cite{delaHarpe}. 
Our formulae are timelike analogues of Phadke's formulae.

\end{document}